%% file: top-ehrhart.tex
\documentclass[reqno,11pt]{amsart}
\usepackage{hyperref}
\usepackage{amssymb,amsmath}
\usepackage{amsthm}
\usepackage{enumerate}
\usepackage{url}
\usepackage{booktabs}
\newtheorem{theorem}{Theorem}
\newtheorem{proposition}[theorem]{Proposition}
\newtheorem{lemma}[theorem]{Lemma}
\newtheorem{definition}[theorem]{Definition}
\newtheorem{corollary}[theorem]{Corollary}
\newtheorem{remark}[theorem]{Remark}
\theoremstyle{definition}
\newtheorem{example}[theorem]{Example}
\newcommand{\C}{{\mathbb C}}
\newcommand{\R}{{\mathbb R}}
\newcommand{\Z}{{\mathbb Z}}
\newcommand{\N}{{\mathbb N}}
\newcommand{\Q}{{\mathbb Q}}

\newcommand{\CH}{{\mathcal H}}

\newcommand{\CL}{{\mathcal L}}
\newcommand{\CM}{{\mathcal M}}

\newcommand{\CV}{{\mathcal V}}
\newcommand{\la}{{\langle}}
\newcommand{\ra}{{\rangle}}

\newcommand{\card}{\operatorname{Card}}
\newcommand{\lin}{\operatorname{lin}}

\newcommand{\vol}{\operatorname{vol}}
\def\binomial(#1,#2){\binom{#1}{#2}}
\def\mult(#1,#2){\binom{#1}{#2}}   % multinomial
\renewcommand{\a}{{\mathfrak{a}}}
\renewcommand{\b}{{\mathfrak{b}}}
\renewcommand{\c}{{\mathfrak{c}}}

\newcommand{\e}{{\mathrm{e}}}

\renewcommand{\t}{{\mathfrak{t}}}
\newcommand{\p}{{\mathfrak{p}}}

\newcommand{\lattice}{\Lambda}
\newcommand{\co}{{k_0}}
\newcommand{\ellpert}{{\ell'}} %% perturbation vector for \ell
\newcommand{\Moebius}{\textup{M\"obius}}
\def\ve#1{#1}    % our vectors are not boldface
\usepackage{ifthen}
\newcommand{\DeclareBracket}[3]{
  \newcommand{#1}[2][]{%
  \ifthenelse%
  {\equal{##1}{}}%
  {\left#2##2\right#3}%
  {\csname ##1l\endcsname#2##2\csname ##1r\endcsname#3}}}    
\DeclareBracket\charfun[]
\newcommand{\Jposet}[2]{\mathcal J^{#1}_{\geq #2}}   % Poset of subsets of #1 of cardinality >= #2
\usepackage{ifpdf,graphics,xcolor}

\newenvironment{inputlist}
  {\begin{enumerate}[\quad\rm({I}$_\bgroup 1\egroup$)]}
  {\end{enumerate}}
\newenvironment{outputlist}
  {\begin{enumerate}[\quad\rm({O}$_\bgroup 1\egroup$)]}
  {\end{enumerate}}
\title[Highest  Ehrhart
  coefficients]{Computation of the highest coefficients\\ of weighted  Ehrhart quasi-polynomials\\ of rational polyhedra}

\author{V. Baldoni}
\address{Velleda Baldoni: Dipartimento di Matematica, Universit\`a degli studi di  Roma ``Tor Vergata'',
Via della ricerca scientifica 1, I-00133 Roma, Italy}
\email{baldoni@mat.uniroma2.it}
\author{N. Berline}
\address{Nicole Berline: Centre de Math\'ematiques Laurent Schwartz, \'Ecole Polytechnique, 91128 Palaiseau Cedex, France}
\email{nicole.berline@math.polytechnique.fr}
\author{J. A. De Loera}
\address{Jes\'us A. De Loera:  Department of
  Mathematics, University of California,
  Davis, One Shields Avenue, Davis, CA, 95616, USA}
\email{deloera@math.ucdavis.edu}
\author{M. K\"oppe}
\address{Matthias~K\"oppe:  Department of
  Mathematics, University of California,
  Davis, One Shields Avenue, Davis, CA, 95616, USA}
\email{mkoeppe@math.ucdavis.edu}
\author{M. Vergne}
\address{Mich\`ele Vergne: Institut de Math\'ematiques de Jussieu, Th{\'e}orie des
  Groupes, Case 7012, 2 Place Jussieu, 75251 Paris Cedex 05, France}
 \email{vergne@math.jussieu.fr}

\usepackage{rcs}
\RCS $Revision: 242 $
\date{Nov 7, 2010 (Revision \RCSRevision)}

\begin{document}
\maketitle

\begin{abstract}
  This article concerns the computational problem of counting the lattice points inside convex polytopes, when each point
  must be counted with a weight associated to it. We describe an efficient algorithm for computing the highest degree
  coefficients of the weighted Ehrhart quasi-polynomial for a rational simple
  polytope in varying dimension, when the weights of the lattice points
  are given by a polynomial function $h$. Our technique  is based on a refinement of an algorithm of 
  A.~Barvinok in the unweighted case (i.e., $h \equiv 1$). 
 % [\emph{Computing the {E}hrhart quasi-polynomial
  %  of a rational simplex}, Math.\ Comp.\ \textbf{75} (2006), pp.~1449--1466].
  In contrast to Barvinok's method, our method is local, obtains an
  approximation on the level of generating functions,  handles the general weighted
  case, and provides the coefficients in closed form as step polynomials of the dilation.  To
  demonstrate the practicality of our approach we report on   
  computational experiments which show even our simple implementation can compete
  with state of the art software.
\end{abstract}

{\scriptsize\tableofcontents}

\section{Introduction}

Computations with lattice points in convex polyhedra arise in various areas of computer science, 
mathematics, and statistics (see e.g., \cite{beckrobins, berichtesurvey,MicciancioGoldwasserbook} and the many references therein). 
Given $\p$, a rational convex polytope in~$\R^d$, and $h(x)$, a polynomial function on $\R^d$ (often called a \emph{weight function}), 
this article considers the important computational problem of computing, or estimating, the sum of the values of $h(x) $ over the lattice 
points belonging to $\p$, namely
$$
S(\p,h)=\sum_ {x\in \p \cap \Z^d}h(x).
$$
The function $S(\p,h)$ has already been studied
extensively in the \emph{unweighted} case, i.e., when $h(x)$ takes
only the constant value $1$ (in that case of course $S(\p,1)$ is just the number of lattice points of $\p$).
Many papers and books have been written about the structure of that function
(see, e.g.,
\cite{barvinokzurichbook,beckrobins} and the many references therein).    Nevertheless, in many applications $h(x)$ can 
  be a much more complicated function. Important
examples of such a situation appear, for instance, in enumerative
combinatorics \cite{andrewspaulereise}, statistics \cite{DG,
  Chenetalstats}, symbolic integration \cite{baldoni-berline-deloera-koeppe-vergne:integration} and non-linear optimization
\cite{deloera-hemmecke-koeppe-weismantel:intpoly-fixeddim}. Still, only 
a small number of algorithmic results exist about the case of  an arbitrary polynomial $h$.

It is well-known that when the polyhedron $\p$ is dilated by an
integer factor $n\in \N$, we obtain a function of $n$, the so-called
\emph{weighted Ehrhart quasi-polynomial} of the pair $(\p,h)$, namely
$$
S(n\p,h)=\sum_ {x\in n\p \cap \Z^d}h(x)= \sum_{m=0}^{d+ M} E_m(n \bmod q) \, n ^m.
$$ 
This is a quasi-polynomial in the sense that the function is a sum
of monomials up to degree $d+M$, where $M=\deg h$, but whose
coefficients $E_m$ are periodic functions of $n$.
The coefficient functions $E_m$ are periodic functions
with period~$q$, where $q\in\N$ is the smallest positive integer  such that $q\p$ is a \emph{lattice polytope}, i.e.,
its vertices are lattice points.
We will make
this more precise later (we
recommend \cite{barvinokzurichbook,beckrobins} for excellent
introductions to this topic). 

To begin realizing the richness of $S(n\p,h)$,
note that its leading highest degree coefficient $E_{d+M}$ (which actually
does not depend on~$n$) is precisely equal to
$\int_\p h(x) \,\mathrm d x$, i.e., the integral of $h$ over the polytope~$\p$, when $h$
is homogeneous of degree~$M$. These
integrals were studied in \cite{Barvinok-1991}, \cite{Barvinok-1992} and more recently at
\cite{baldoni-berline-deloera-koeppe-vergne:integration}. Still most other coefficients are
difficult to understand, even for easy polytopes, such as simplices (see \cite{beifangchen:2002} for
a survey of results and challenges).
The key aim of this article is to achieve the fast computation of the
first few top-degree (weighted) coefficients $E_m$ via
an approximation of $S(n\p,h)$ by a quasi-polynomial
that shares the highest coefficients with $S(n\p,h)$.\smallbreak

We now explain the known results achieved so far in the literature. 
It is important to stress that computing \emph{all} the
coefficients $E_m$ for $m=0,\dots, d+M$ is an NP-hard problem, thus the
best one can hope for theoretical results is to obtain an \emph{approximation}, as we propose to do here.
Until now most results dealt only with the \emph{unweighted case}, i.e., $h(x)=1$ and we
summarize them here:  A. Barvinok first obtained for lattice
polytopes~$\p$ 
a polynomial-time algorithm
that for a fixed integer $k_0$ can compute the highest $k_0$
coefficients $E_m$ (see~\cite{barvinok-1994:ehrhart}). For this he used
Morelli's identities \cite{morelli} and relied on an oracle that
computes the volumes of faces. 

Later, in \cite{barvinok-2006-ehrhart-quasipolynomial},  
Barvinok obtained a formula relating the $k$ highest degree  coefficients of the (unweighted) Ehrhart
quasi-polynomial of a rational polytope to volumes of sections of
the polytope by certain affine lattice subspaces of codimension $< k$.  As a consequence, he proved that the
$k$ highest degree coefficients of the unweighted Ehrhart
quasi-polynomial of a \emph{rational simplex} can be computed by a
polynomial algorithm, when the dimension $d$ is part of the input, but $k$ is fixed.
More precisely, given a \emph{dilation class} $n\bmod q$ with $n\in\N$, Barvinok's algorithm
computes the \emph{numbers} $E_m(n \bmod q)$ by an interpolation technique.
However, neither a closed formula for these $E_m$, depending on $n$, nor a
generating function for the coefficients became available
from~\cite{barvinok-2006-ehrhart-quasipolynomial}.  In fact,
in that article \cite[Section~8.2]{barvinok-2006-ehrhart-quasipolynomial} 
the question of efficiently computing such a closed form expression was raised.

A key point of both Barvinok's and our method is the following.
The  sum $S(\p,h)$ has  natural generalizations, the \emph{intermediate} sums
$S^L(\p,h)$, where $L\subseteq V = \R^d$ is a rational vector subspace. For
a polytope $\p\subset V$ and a polynomial $h(x)$
$$
S^L(\p,h)= \sum_{x} \int_{\p\cap (x+L)} h(y)\,\mathrm dy,
$$
where the summation index $x$ runs over the projected lattice in
$V/L$. In other words, the polytope $\p$ is sliced along lattice
affine subspaces parallel to $L$ and the integrals of $h$ over the
slices are added up. For $L=V$, there is only one term and
$S^V(\p,h)$ is just the integral of $h(x)$ over $\p$, while, for
$L=\{0\}$, we recover $S(\p,h)$. Barvinok's method in~\cite{barvinok-2006-ehrhart-quasipolynomial} was to  introduce
particular linear combinations of the intermediate sums,
$$
\sum_{L\in\CL}\lambda(L)S^L(\p,h).
$$
It is natural to replace   the polynomial  weight $h(x)$
with an exponential function $x\mapsto \e^{\la \xi,x\ra}$, and consider the
corresponding holomorphic functions of $\xi$ in the dual $V^*$.
Moreover, one can allow $\p$ to be unbounded, then the sums
$$
S^L(\p)(\xi) = \sum_{x} \int_{\p\cap (x+L)} \e^{\la \xi,y\ra}\,\mathrm dy
$$
still make sense as meromorphic functions on $V^*$. The map $\p\mapsto
S^L(\p)(\xi)$ is a valuation. 

In
\cite{Baldoni-Berline-Vergne-2008}, it was  proved that 
a version of Barvinok's construction \emph{on the level of generating
  functions}, namely 
$\sum_{L\in\CL}\lambda(L)S^L(\p)(\xi)$, approximates $S(\p)(\xi)$ in
a certain ring of meromorphic functions (a precise statement is
given below). The proof in \cite{Baldoni-Berline-Vergne-2008} relied
on the Euler-Maclaurin expansion of these functions. Another proof,
using the Poisson summation formula, will appear in
\cite{berline-brion-vergne:poisson}. 

In the present article, we introduce a \emph{simplified} way to approximate
$S(\p)(\xi)$ for the case of a simplicial affine cone $\p=s+\c$, 
which levels the way for a practical and  efficient implementation.
Via Brion's
theorem, it is sufficient to sum up these \emph{local} contributions of the
tangent cones of the vertices. 
We present a method for computing the highest degree coefficients of the Ehrhart
quasi-polynomial of a rational simple polytope, by applying the approximation
theorem to each of the cones at vertices of $\p$. 
The complexity depends  on the number of vertices of the
polytope, and thus if the simple polytope is presented by its vertices (rather than
by linear inequalities), we obtain a polynomial-time algorithm. 
In particular, the algorithm is polynomial-time for the case of a simplex.
We obtain the Ehrhart coefficient functions $E_m(n \bmod q)$ in a \emph{closed
  form} as \emph{step polynomials}, i.e., polynomials in $n$ whose
coefficients are 
modular expressions $(\zeta_i n) \bmod q_i$ with integers $\zeta_i$ and~$q_i$,
for example $(2 n) \bmod 3$.
Having a closed formula available considerably strengthens Barvinok's result
in~\cite{barvinok-2006-ehrhart-quasipolynomial} even in the
unweighted case~$h=1$.\smallbreak

The structure of this paper is as follows.
In Section~\ref{notation}, we first present some necessary preliminaries.
Section~\ref{s:two-key-ideas} explains the
intermediate generating function~$S^L(\p)(\xi)$ in more detail.  Then we show how to
use a grading of~$S(\p)(\xi)$ to extract the highest degree coefficients of the
weighted Ehrhart polynomial in the case of a lattice polytope.  This motivates
the approximation results for generating functions.
In Section~\ref{section-simplicial-cone}, 
we give  a simple  
proof of the approximation theorem of \cite{Baldoni-Berline-Vergne-2008}, in the
case of  a simplicial cone (see Theorem \ref{approx-generating-function}). 
The theorem uses the notion of a \emph{patching function} (essentially a 
form of M\"obius inversion formulas described in Subsection
\ref{patchingfun}).  We exhibit an explicit and easily computable such
patching function. 
Using these tools, we show in Section~\ref{section-patchedgenfun-computation}
that the approximation for a cone~$s+\c$ (on the level of generating functions) can be computed
efficiently as a closed formula.  The formula makes the periodic dependence on the
vertex~$s$ explicit.
Finally, in Section~\ref{section-Ehrhartcomputation}, we give the
polynomial-time algorithm to compute the coefficients $E_m(n \bmod q)$ 
as step polynomials.
Our main result (Theorem~\ref{powersoflinear-stepfun}) says that, for every
fixed number~$k_0$, there exists a polynomial-time algorithm that, given a simple
polytope $\p$ of arbitrary dimension, a linear form $\ell\in V^*$, a
nonnegative integer~$M$, computes the highest $k_0+1$ coefficients
$E_{M+d-k_0},\dots,E_{M+d}$ of the weighted Ehrhart quasi-polynomial
$S(n\p,h=\ell^M)$ 
in the form of step polynomials.

Four comments are in order about the applicability and potential practicality of the main results: First, although the weight~$h$ used in
Theorem \ref{powersoflinear-stepfun} is a power of a linear form, as is carefully explained in \cite{baldoni-berline-deloera-koeppe-vergne:integration},  
one can obtain similar complexity of computation for polynomials that depend on a fixed number of variables, or with fixed degree
(Corollary~\ref{ehrhart-barvinok-style-general}). Second, it is also worth noting that  using perturbations (see e.g., \cite{edelsbrunner-perturbation}), 
triangulations \cite{DRStriangbook} or  simplicial cone decompositions of polyhedra (see, e.g., \cite{koeppe:irrational-barvinok}), 
one can extend computations from simple polytopes to arbitrary polytopes. Third, since our approximation is done at the
level of generating functions, it extends the complexity result from \cite{barvinok-2006-ehrhart-quasipolynomial} to the weighted case.
Finally,  at the end of the article we report on experiments using a simple implementation of the algorithm in \emph{Maple}, demonstrating 
it  is competitive with more sophisticated software tools. 
This indicates a potential to use this algorithm for experimentally verifying conjectures on
the positivity of the Ehrhart coefficients of certain polytopes, for examples where the computation of the full Ehrhart polynomials is out of reach.
The algorithms presented here require a rich mixture of computational geometry and  algebraic-symbolic computation.

\section{Preliminaries} \label{notation}

\subsection{Rational convex polyhedra}
We consider  a \emph{rational vector space} $V$ of dimension $d$, that is
to say a finite dimensional real vector space with a lattice denoted
by  $\lattice$.   We will need to consider subspaces and quotient
spaces of $V$, this is why we cannot simply let $V=\R ^d$ and
$\lattice = \Z^d$. 
A point $v\in V$ is called \emph{rational} if there exists a non-zero integer 
$q$ such that $qv\in \lattice$. The set  of rational points
in $V$ is denoted by $V_\Q$.
A subspace $L$ of $V$ is called
\emph{rational} if $L\cap \lattice $ is a lattice in $L$. If $L$ is a
rational subspace, the image of $\lattice$ in $V/L$ is a lattice in
$V/L$, so that $V/L$ is a rational vector space. The image of
$\lattice$ in $V/L$ is called the \emph{projected lattice}.
A rational space $V$, with lattice $\lattice$, has a canonical
Lebesgue measure $\mathrm dx = \mathrm dm_\Lambda(x)$, for which  $V/\lattice$ has measure~$1$.

A \emph{convex rational polyhedron} $\p$ in $V$ (we will simply say
 \emph{polyhedron}) is, by definition, the intersection of a finite number of
closed half spaces bounded by  rational affine hyperplanes. We say
that  $\p$ is \emph{full-dimensional} (in $V$) if the affine span of $\p$ is $V$.

In this article, a \emph{cone} is a polyhedral cone (with vertex $0$) and
an affine cone is a translated set $s+\c$ of a cone $\c$. A  cone
$\c$ is called \emph{simplicial} if it is  generated by independent
elements of $V $. A simplicial cone~$\c$ is called \emph{unimodular} if it
is generated by independent integral vectors $v_1,\dots, v_k$ such
that $\{v_1,\dots, v_k\}$ can be completed to an integral basis of
$V$. An affine cone~$\a$ is called simplicial (respectively, simplicial
unimodular) if the associated cone is.  A \emph{polytope} $\p$ is a  compact polyhedron.
 The set of vertices of $\p$ is denoted by $\CV(\p)$.
 For each vertex $s$, the cone of feasible directions at $s$  is denoted by
$\c_s$. For details in all these notions see, e.g., \cite{barvinokzurichbook}.

\pagebreak[5]

\subsection{Generating functions: Exponential sums and integrals}
\begin{definition}
We denote by $\CH(V^*)$ the ring of holomorphic functions defined
around $0\in V^*$. We denote by $\CM(V^*)$ the ring of meromorphic
functions defined around $0\in V^*$  and by $\CM_{\ell}(V^*)\subset
\CM(V^*)$ the subring consisting of  meromorphic functions
$\phi(\xi)$ which can be written as a quotient of a holomorphic
function and a product of linear forms.
\end{definition}
This paper relies on the study of important examples
of functions in $\CM_{\ell}(V^*)$, the
following  continuous and discrete generating functions $I(\p,\Lambda)$ and
$S(\p,\Lambda)$ associated to a convex polyhedron $\p$. Both have an important
additivity property which makes them valuations (see \cite[Chapter
8]{barvinokzurichbook} or the survey \cite{BarviPom} for a detailed presentation, here
we summarize the essentials).  

\begin{definition}
Let $M$ be a vector space. A \emph{valuation} $F$ is a  map
from the set of polyhedra $ \p\subset V$ to the vector space $M$
such that whenever the indicator functions $\charfun{\p_i}$ of a
family of polyhedra $\p_i$ satisfy a linear relation $\sum_i r_i
\charfun{\p_i}=0$, then the elements $F(\p_i)$ satisfy the same
relation
$ \sum_i r_i F(\p_i)=0$.
\end{definition}
\begin{proposition}\label{valuationI}
There exists a unique valuation  $I(\cdot,\Lambda)$ which  associates  to every polyhedron
$\p\subset V$ a meromorphic function $I(\p,\Lambda)\in
\CM_{\ell}(V^*)$, so that the following properties hold:

\begin{enumerate}[\rm(i)]
\item 
 If the polyhedron $\p$ is not full-dimensional or if $\p$ contains a straight line,
  then $I(\p,\Lambda)=0$.
\item If $\xi\in V^*$ is such that $\e^{\la \xi,x\ra}$ is integrable over $\p$,
  then
$$
I(\p,\Lambda)(\xi)= \int_\p \e^{\la \xi,x\ra} \,\mathrm dm_\Lambda(x).
$$

\item For every point $s\in V_{\Q}$, one has
$$
I(s+\p,\Lambda)(\xi) = \e^{\la \xi,s\ra}I(\p, \Lambda)(\xi).
$$
\end{enumerate}
\end{proposition}
We will call $I(\p,\Lambda)(\xi)$ the \emph{continuous generating function} of~$\p$.
\begin{proposition}\label{Svaluation}
There exists a unique valuation  $S(\cdot,\Lambda)$ which  associates to every polyhedron
$\p\subset V$ a meromorphic function $S(\p,\Lambda)\in
\CM_{\ell}(V^*)$, so that the following properties hold:

\begin{enumerate}[\rm(i)]
\item If $\p$ contains a straight line, then $S(\p,\Lambda)=0$.

\item If $\xi\in V^*$ is such that $\e^{\la \xi,x\ra}$ is summable over the set
  of lattice points of $\p$, then
$$
S(\p,\Lambda)(\xi)= \sum_{x\in \, \p\cap \lattice} \e^{\la \xi,x\ra} .
$$

\item For every point $s\in\lattice $, one has
$$
S(s+\p,\Lambda)(\xi) = \e^{\la \xi,s\ra}S(\p,\Lambda)(\xi).
$$
\end{enumerate}
\end{proposition}
 $S(\p,\Lambda)(\xi)$ is called the \emph{(discrete) generating function} of $\p$.\smallbreak

\subsection{Brion's theorem}
A consequence of the valuation property is the following fundamental
theorem. It follows from the Brion-Lawrence-Varchenko decomposition of a
polyhedron into the supporting cones at its vertices \cite{Brion88,barvinokzurichbook};
see also \cite{brionvergne1997}, Proposition 3.1,  for a more
general Brianchon-Gram type identity.
\begin{theorem} \label{brion}Let $\p$ be a polyhedron with set of vertices $\CV(\p)$. For each
vertex~$s$, let $\c_s$ be the cone of feasible directions at $s$.
Then
$$
S(\p,\Lambda)=\sum_{s\in \CV(\p)}S(s+\c_s,\Lambda).
$$
\end{theorem}

\subsection{Notations and basic facts in the case of  a simplicial
cone}\label{simplicial}

For all of the notions below see \cite{barvinokzurichbook}.
 Let $v_i\in\Lambda,i=1,\dots, d$ be linearly independent integral
vectors and let $\c=\sum_{i=1}^d \R_+ v_i$ be the cone they span.

\begin{definition}
The \emph{fundamental parallelepiped} $\b$ of the cone (with respect to
the generators  $v_i,i=1, \dots,d$) is the set
\begin{equation*}\label{cell}
\b=\sum_{i=1}^d [0,1\mathclose[\, v_i.
\end{equation*}
\end{definition}
Note that the set has a half-open boundary. We immediately have:
\begin{lemma}\label{lemma-integral}
  Let $s\in V$. 
  Then
  \begin{equation}\label{I}
    I(s+\c,\Lambda)(\xi)=\e^{\la \xi,s \ra} \frac{(-1)^d
      \vol_\Lambda(\b)} {\prod_{i=1}^d\la \xi,v_i\ra},
  \end{equation}
  where $\vol_\Lambda(\b)$ is the volume of the fundamental parallelepiped
  with respect to the Lebesgue measure $\mathrm dm_\Lambda$ defined by the
  lattice.
\end{lemma}
If $V=\R^d$ and $\Lambda=\Z^d$, then $\vol_\Lambda(\b) = \mathopen|\det(v_1,\dots,v_d)|$, and so
\begin{equation}\label{I-with-det}
  I(s+\c,\Z^d)(\xi)=\e^{\la \xi,s \ra} \frac{(-1)^d
    \mathopen| \det(v_1,\dots,v_d)|} {\prod_{i=1}^d\la \xi,v_i\ra}.
\end{equation}
We also recall the following elementary but  crucial lemma.
\begin{lemma}\label{lemma-cell}
  \begin{enumerate}[\rm(i)]
  \item  The affine cone $(s+\c)\,\cap \lattice$ is the disjoint union
    of the translated parallelepipeds $s +\b + v$, for $v\in \sum_{j=1}^d \N v_j$.

  \item  The set of lattice points in the affine cone $s+\c$ is the
    disjoint union of the sets $x+\sum_{i=1}^d \N v_i$ when $x$ runs over the
    set $(s+\b)\,\cap \lattice$.

  \item  The number of lattice points in the parallelepiped $s+\b$ is equal to
    the volume of the parallelepiped with respect to the Lebesgue measure
    $\mathrm dm_\Lambda$ defined by the
    lattice, that is
    $$ \card((s+\b)\,\cap \lattice)= \vol_\Lambda(\b).$$
    In particular, when $V=\R^d$ and $\Lambda=\Z^d$, then
    $$ \card((s+\b)\,\cap \Z^d) = \mathopen|\det(v_1,\dots,v_d)|.
 $$
\end{enumerate}
\end{lemma}
\medskip

The study of  the generating function $S(s+\c,\Lambda)(\xi)$ of the affine cone $s+\c$
will be a crucial tool. It relies on expressing
$S(s+\c,\Lambda)(\xi)$ in terms of the  generating function
 $S(s+\b,\Lambda)(\xi)=\sum_{x\in (s+\b)\, \cap \lattice}\e^{\la \xi,x
\ra}$ of the fundamental parallelepiped.  Lemma \ref{lemma-cell} (ii)
immediately gives:
\begin{lemma}\label{lemma-cell-generating_function}
\begin{equation}\label{Sbox}
S(s+\c,\Lambda)(\xi)=S(s+\b,\Lambda)(\xi) \frac{1}{\prod_{j=1}^d(1- \e^{\la \xi,v_j\ra})}.
\end{equation}
\end{lemma}

{\small
\begin{example}
  Consider the case where $V=\R$ and $\lattice=\Z$. Let $\c=\R_+$.
  Let $s\in \R$, then the ``fractional part'' $\{s\}\in [0,1[$ is defined as
  the unique real number  such that $s-\{s\}\in \Z$. 
  Then the unique integer~$\bar s$ in~$s+\b$ is 
  $ %%\lceil s\rceil =  %%% not defined...
  s+ \{-s\}$, and
  so equations \eqref{I} and \eqref{Sbox} give
  $$I(s+\c,\Z)(\xi)=\e^{\xi s}\frac{-1}{\xi}\quad\text{and}\quad
  S(s+\c,\Z)(\xi)= \e^{\xi (s+ \{-s\})}\frac{1}{1-\e^{\xi}}.$$
\end{example}}

\section{Key ideas of the approximation theory}
\label{s:two-key-ideas}

\subsection{Weighted Ehrhart quasi-polynomials}\label{Weighted Ehrhart quasi-polynomials}

Let $\p\subset V $ be a rational polytope and let $h(x)$ be a
polynomial function of degree $M$ on $V$.  We consider the following
weighted sum over the set of lattice points of $\p$,
$$
\sum_ {x\in \p\cap \lattice}h(x).
$$
When $\p$ is dilated by a non-negative integer $n\in \N$, we obtain the
weighted Ehrhart quasi-polynomial of the pair $(\p,h)$.
\begin{definition}
  Let $q$ be the smallest positive integer such
  that $q\p$ is a lattice polytope.  The we define the Ehrhart
  quasi-polynomial $ E(\p,h; n)$ and its coefficients $E_m(\p,
  h; n\bmod q)$ by
  $$
  E(\p,h; n) = \sum_ {x\in n\p\cap \lattice }h(x)= \sum_{m=0}^{d+ M} E_m(\p, h; n\bmod q) \, n ^m.
  $$
\end{definition}
We note that the coefficients $E_m$ depend on $n$, but they actually depend only on
$n\bmod q $, where  $q$ is the smallest positive integer such
that $q\p$ is a lattice polytope. If  $h(x) $ is homogeneous of
degree $M$, the highest degree coefficient $E_{d+M}$ is equal to the
integral $\int_\p h(x) \,\mathrm dx$ (see
\cite{baldoni-berline-deloera-koeppe-vergne:integration} and references
therein). \smallbreak

We concentrate on the special case where the polynomial $h(x)$ is a
power of a linear form
$$
h(x)=\frac{\langle\xi,x\rangle^M}{M!}.
$$
This is not a restriction because any polynomial can be written as a linear combination of
powers of linear forms.  In fact, as discussed in~\cite{baldoni-berline-deloera-koeppe-vergne:integration}, whenever the
polynomial~$h(x)$ is either of fixed degree or only depends on a fixed number
of variables (possibly after a linear change
of variables), then only a
polynomial number of powers of linear forms are needed, and such a
decomposition can be computed in polynomial time.
We introduce the following notation.

\begin{definition}\label{power-linear_form_top_Ehrhart}
  Let $q$ be as above. We define  the Ehrhart quasi-polynomial $E(\p,\xi, M; n)$ and the  coefficients $E_m(\p, \xi,M; n \bmod q)$ for $m=0,\ldots, M+d$ by
  $$
  E(\p,\xi, M; n) = 
  \sum_ {x\in n\p\cap \lattice} \frac{\langle\xi,x\rangle^M}{M!}=
  \sum_{m=0}^{M+d}E_m(\p, \xi,M; n\bmod q)\,n^{m}.
  $$
\end{definition}

It will be convenient in this paper to introduce the following
notations.
For a positive integer $q\in \N$ and a real number ~$n\in\R$, we write
\begin{displaymath}
  \lfloor n\rfloor_q := q\bigl\lfloor \tfrac1q n\bigr\rfloor \in q\Z,\quad
  \{n\}_q := (n \bmod q) \in [0,q),
\end{displaymath}
which give the unique decomposition
\begin{displaymath}
  n = \lfloor n\rfloor_q + \{n\}_q.
\end{displaymath}
By $\lfloor n\rfloor := \lfloor n\rfloor_1$ and $\{n\} := \{n\}_1$ we obtain
the ordinary ``floor'' and ``fractional part'' notations.  Finally, $\lceil
n\rceil := -\lfloor -n\rfloor$ is the ``ceiling'' notation.\smallbreak

{\small
\begin{example}
\begin{figure}[t]
  \centering
  \ifpdf
  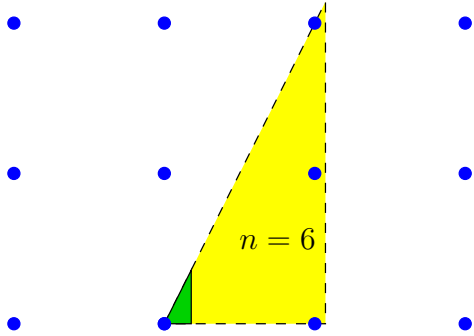
   \else
   \input{triangle528.pstex_t}
   \fi
  \caption{The example triangle~$\t$ and its dilation~$6\t$}
  \label{fig:triangle528}
\end{figure}
Consider the rational triangle $\t$ with vertices $(0,0)$, $(\frac{5}{28},0)$, and $(\frac{5}{28},\frac{5}{14})$ as shown in Figure \ref{fig:triangle528}.
Let us compute $E(\t,\xi,M;n)$ 
for this small example.
Note that the integer $q$ such that $q\t$ is a lattice polytope is $q=28$.

  In what follows consider powers of the linear form $\xi=x+y$ as weights
for the lattice points. When the power $M=0$,  then we obtain a constant weight
and  the quasi-polynomial $E(\t,\xi,0;n)$   counts the lattice points inside the various dilations of $\t$:  it is given by the following formula:
\[
{\frac {25}{784}}\,{n}^{2}+ \left( -{\frac {5}{392}} \lbrace 5\,n \rbrace_{28} +{\frac {5}{14}} \right) \,n+\left( 1+ {\frac {1}{784}} \, \left( \lbrace 5\,n \rbrace_{28}  \right) ^{2}-\frac{1}{14}\,\lbrace 5\,n\rbrace_{28} \right).
\]
Indeed, when $n=1$ (no dilation) there is only one lattice point and the formula  above reduces to
\[
{\frac {1089}{784}}-{\frac {33}{392}}\,\{5\}_{28} +{\frac {1}{784}}\, \left( \{5\}_{28}  \right) ^{2}= \, {\frac {1089}{784}}-{\frac {33}{392}}\,5 +{\frac {1}{784}}\, 25=1.
\]
When we dilate the same triangle six times, i.e., $n=6$, we obtain four lattice points.
\[
{\frac {841}{196}}-{\frac {29}{196}}\, \lbrace 30 \rbrace_{28} +{\frac {1}{784}}\, (\lbrace 30 \rbrace_{28} ) ^{2}=\frac{841}{196}-\frac{29}{196}\cdot2+\frac{1}{784}\cdot4=4.
 \]
Next let us  take $M=1$, in that case  the lattice point
$(a,b)$ is counted with weight $a+b$. In this case the top coefficient is equal to the integral of the linear form $\xi=x+y$ over $\t$.
The quasi-polynomial $E(\t,\xi,1;n)$ is given by the following formula:
\begin{multline*}
{\frac {125}{16464}}\,{n}^{3}
+\left({-\frac {25}{5488}} \{ 5\,n\}_{28} +{\frac {75}{784}} \right) \,{n}^{2}
+\left( {\frac {5}{5488}} \left(  \{ 5\, n \}_{28}  \right) ^{2} -{\frac {15}{392}}\, \{ 5\,n \}_{28} +{\frac {25}{84}} \right) \, n\\
+\left( -{\frac {1}{16464}}\, \left( \{5\,n\}_{28} \right) ^{3}
 +{\frac {3}{784}}\, \left( \{5\,n\}_{28}  \right) ^{2}
 -{\frac {5}{84}}\, \{5\,n\}_{28} \ \right).
\end{multline*}
Substitute again $n=1$ in the expression, to obtain
\[
{\frac {275}{686}}-{\frac {1685}{16464}}\, \{ 5 \}_{28} +{\frac {13}{2744}}\, \left( \{5\}_{28}  \right) ^{2}-{\frac {1}{16464}}\, \left(  \{ 5 \}_{28} \right) ^{3}=0.
\]
Note that since only the lattice point $(0,0)$ lies within the triangle at $n=1$, the quasipolynomial must evaluate to zero.
\end{example}}

In practice, it is impossible to compute $E(\p,\xi,M;n)$ except when $\p$ is of small dimension (and $M$ relatively small).
Thus we restrict our ambitions:

Let us fix a number $k_0$. Our goal will be to compute the $k_0+1$
highest degree coefficients $E_m(\p, \xi,M; n \bmod q)$, for $m= M+d, \dots,
M+d-k_0$.
We will  be able to give a polynomial time algorithm to do so.

\subsection{Grading of the generating functions}

A key property that we will make use of in this article is the following \emph{grading}
of~$\CM_{\ell}$:  A function $\phi(\xi)\in \CM_{\ell}(V^*)$ has a unique expansion
into homogeneous rational functions
$$
\phi(\xi)= \sum_{m\geq m_0}\phi_{[m]}(\xi),
$$
where the summands $\phi_{[m]}(\xi)$ have degree $m$ as we define now:
If $P$ is a homogeneous polynomial on $V^*$ of degree $p$, and $D$ a
product of $r$ linear forms, then $\frac{P}{D}$ is an element in
$\CM_{\ell}(V^*)$  homogeneous of degree $m=p-r$. For instance,
$\smash{\frac{\xi_1}{\xi_2}}$ is homogeneous of degree $0$. On this example
we observe that a function in $\CM_\ell(V^*)$ which has non-negative
degree terms need not be  analytic.

In particular, consider the generating function $S(s+\c,\Lambda)(\xi)$ for a
simplicial cone~$\c$.  By Lemma~\ref{lemma-cell-generating_function}, 
\begin{equation*}
  S(s+\c,\Lambda)(\xi)=S(s+\b,\Lambda)(\xi) \frac{1}{\prod_{j=1}^d(1- \e^{\la \xi,v_j\ra})}.
\end{equation*}
Thus, $S(s+\c,\Lambda)\in \CM_{\ell}(V^*)$, and so it admits a
decomposition into homogeneous components:
\begin{lemma}\label{lemma-cell-generating_function-grading}
  \begin{equation}\label{homogeneous-components}
    S(s+\c,\Lambda)( \xi)=  S(s+\c,\Lambda)_{[-d]}(\xi)+ S(s+\c,\Lambda)_{[-d+1]}(\xi) + \cdots,
  \end{equation}
  and the lowest degree term $S(s+\c,\Lambda)_{[-d]}(\xi)$ is equal to 
  $I(\c,\Lambda)(\xi)$, i.e., the integral over the unshifted cone~$\c$.
\end{lemma}
\begin{proof}
  We write
\begin{equation}\label{todd}
\prod_{j=1}^d \frac{1}{1- \e^{\la \xi,v_j\ra}}=\prod_{j=1}^d\frac{\la
\xi,v_j\ra}{1- \e^{\la \xi,v_j\ra}}\frac{1}{\prod_{j=1}^d \la
\xi,v_j\ra}.
\end{equation}
 The function $\frac{x}{1-\e^x}$ is
 holomorphic with value $-1$ for $x=0$. Thus we have $S(s+\c,\Lambda)\in \CM_\ell(V^*)$.
 The value at $\xi=0$ of the sum over the parallelepiped is the number of lattice
 points of the parallelepiped, that is $\vol_\Lambda(\b)$. This proves
 the last assertion.
\end{proof}

\subsection{Sketch of the method for lattice polytopes}

We will now explain the key point
of our method, with the simplifying assumption that the vertices of
the polytope are lattice points. We will show that the highest
degree coefficients of the weighted Ehrhart polynomial can be read
out from an approximation of the generating functions of the cones at
vertices. In Section~\ref{section-simplicial-cone} we will study
this approximation, and in Section~\ref{section-patchedgenfun-computation}
we will show how to efficiently compute it.
Then, in Section
\ref{section-Ehrhartcomputation}, we will come back to the
computation of  Ehrhart coefficients for the general case of rational polytopes.

\begin{proposition}\label{top-Ehrhart-with-generating-function}
Let $\p$ be a lattice polytope. Then, for $k\geq 0$, we have
\begin{equation}\label{ehrhart-coeff}
E_{M+d-k}(\p, \xi,M)= \sum_{s\in \CV(\p)}\frac{\langle\xi,s\rangle
^{M+d-k}}{(M+d-k)!} S(\c_s)_{[-d+k]}(\xi).
\end{equation}
The highest degree coefficient is just the integral
$$E_{M+d}(\p,\xi,M)= \int_\p \frac{\langle\xi,x\rangle^M}{M!} \,\mathrm dx.$$
\end{proposition}
\begin{remark}
As functions of $\xi$, the coefficients $E_m(\p, \xi,M)$ are
polynomial, homogeneous of degree $M$. However, in \eqref{ehrhart-coeff}, they are  expressed
as linear combinations of rational functions of $\xi$, whose poles
cancel out.
\end{remark}
\begin{proof}[Proof of Proposition~\ref{top-Ehrhart-with-generating-function}]
 The starting point is Brion's formula.  As the vertices are lattice
points, we have
\begin{equation}\label{generating-function}
\sum_{x\in \,\p\cap \lattice} \e^{\langle\xi,x\rangle}=\sum_{s\in
\CV(\p)}S(s+\c_s)(\xi)=\sum_{s\in
\CV(\p)}\e^{\langle\xi,s\rangle}S(\c_s)(\xi).
\end{equation}
When $\p$ is replaced with $n\p$, the vertex $s$ is replaced with
$ns$ but the  cone $\c_s$ does not change. We obtain
\begin{equation*}\label{generating-function-lattice-vertices-2}
\sum_{x\in n\p\cap \lattice} \e^{\langle\xi,x\rangle}=\sum_{s\in
\CV(\p)}\e^{n \langle\xi,s\rangle}S(\c_s)(\xi).
\end{equation*}
We replace $\xi $ with $t\xi$,
\begin{equation*}\label{generating-function-lattice-vertices-3}
\sum_{x\in n\p\cap \lattice} \e^{t\langle\xi,x\rangle}=\sum_{s\in
\CV(\p)}\e^{nt \langle\xi,s\rangle}S(\c_s)(t \xi).
\end{equation*}
The decomposition into homogeneous components gives
\begin{equation*}
S(\c_s)(t \xi)= t^{-d}I(\c_s)(\xi) + t^{-d+1}S(\c_s)_{[-d+1]}(\xi) +
\cdots + t^k S(\c_s)_{[k]}(\xi)+ \cdots .
\end{equation*}
Hence, the $t^M$-term in  the right-hand side of the above equation is equal to
$$
\sum_{k=0}^{M+d}(nt)^{M+d-k}\, t^{-d+k}\frac{ \langle\xi,s\rangle
^{M+d-k}}{(M+d-k)!} S(\c_s)_{[-d+k]}(\xi).
$$
Thus we have
\begin{multline}\label{generating-function-lattice-vertices-4}
\sum_{x\in n\p\cap \lattice} \frac{\langle\xi,x\rangle ^M}{M! } =
\sum_{s\in \CV(\p)} n^{M+d} \frac{\langle\xi,s\rangle
^{M+d}}{(M+d)!}I(\c_s)(\xi) \\
 + n^{M+d-1}\frac{\langle\xi,s\rangle
^{M+d-1}}{(M+d-1)!}S(\c_s)_{[-d+1]}(\xi)+\cdots +
S(\c_s)_{[M]}(\xi).
\end{multline}
From this relation, we read immediately that $\sum_{x\in n\p\,\cap
\lattice} \frac{\langle\xi,x\rangle ^M}{M! }$ is a polynomial
function of $n$ of degree $M+d$, and that  the coefficient of
$n^{M+d-k}$ is given by (\ref{ehrhart-coeff}).
 The highest degree coefficient is given by
$$
E_{M+d}(\p,\xi,M)= \sum_{s\in \CV(\p)}\frac{\langle\xi,s\rangle
^{M+d}}{(M+d)!}I(\c_s)(\xi).
$$
Applying Brion's formula for the integral, this is equal to
 the term of $\xi$-degree $M$ in $I(\p)(\xi)$, which is indeed the integral
$ \int_\p \frac{\langle\xi,x\rangle^M}{M!} \,\mathrm dx$.
\end{proof}
From Proposition \ref{top-Ehrhart-with-generating-function}, we draw
an important consequence:  in order to compute the $k_0+1$ highest
degree terms of the weighted Ehrhart polynomial for the weight
$h(x)=\frac{\langle\xi,x\rangle^M}{M!}$, we \emph{only need} the
$k_0 + 1$ lowest degree homogeneous  terms of the meromorphic
function $S(\c_s)(\xi)$, for every vertex $s$ of $\p$. We compute
such an approximation in Section~\ref{section-simplicial-cone}; 
it turns out to be sufficient also
in the general case of a rational polytope.

\subsection{Intermediate generating functions}

To obtain the approximation, we study generating functions
which interpolate between the integral $I(\p,\Lambda)$ and the discrete sum
$S(\p,\Lambda)$. This trend of ideas was first discussed by Barvinok in
\cite{barvinok-2006-ehrhart-quasipolynomial}. 
Let $L$ be a rational subspace of $V$.
 To any polyhedron~$\p$ we  associate a
meromorphic function $S^L(\p,\Lambda)(\xi)\in \CM(V^*)$, which is, roughly
speaking, obtained by slicing $\p$ along affine subpaces parallel to~$L$ 
through lattice points, and adding the integrals of $\e^{\la
\xi,x\ra} $ along the slices. Recall that the quotient space $V/L$
is endowed with the projected  lattice $\lattice_{V/L}$. 

\begin{proposition}\label{valuationSL}
Let $L\subseteq V$ be a rational subspace. There exists a unique
valuation  $S^L(\cdot,\Lambda)$ which to every rational polyhedron $\p\subset V$
associates a meromorphic function with rational coefficients
$S^L(\p,\Lambda)\in \CM(V^*)$ so that the following properties hold:

\begin{enumerate}[\rm(i)]
\item  If $\p$ contains a line, then $S^L(\p,\Lambda)=0$.

\item
  \begin{equation}\label{SL}
    S^L(\p,\Lambda)(\xi)= \sum_{x\in \lattice_{V/L}} \int_{\p\cap (x+L)} \e^{\la
      \xi,y\ra}\,\mathrm dy,
  \end{equation}
  for every $\xi\in V^*$ such that the above sum converges.

\item For every point $s\in \lattice$, we have
$$
S^L(s+\p,\Lambda)(\xi) = \e^{\la \xi,s\ra}S^L(\p,\Lambda)(\xi).
$$
\end{enumerate}
\end{proposition}
We call the function $S^L(\p,\Lambda)$ an \emph{intermediate generating function}.
The proof is  entirely analogous to the case  $L= \{0\}$, see
Theorem 3.1 in \cite{BarviPom}, and we omit it.

For $L=\{0\}$, we recover the valuation $S$.  For $L=V$, we have
$S^V(\p,\Lambda)=I(\p,\Lambda)$.  In particular, if $\p$ is not full-dimensional, then $S^V(\p,\Lambda)=0$.\smallbreak 

If $\p$ is compact, the meromorphic function $S^L(\p,\Lambda)(\xi)$ is actually regular
at $\xi=0$, and its value for $\xi=0$ is the $\Q$-valued valuation
$E_{L^\perp}(\p)$ considered by Barvinok
\cite{barvinok-2006-ehrhart-quasipolynomial}.

\begin{remark} The function $S^L(\p,\Lambda) $ is actually an element of
  $\CM_{\ell}(V^*)$, just like the functions
  $S(\p,\Lambda)$ and $I(\p,\Lambda)$.  This follows from an interesting decomposition that
  allows to write $S^L$ as a combination of terms using $S$ and~$I$ for
  certain cones.  This and other properties of the
  valuation~$S^L(\cdot,\Lambda)$ will be discussed in a forthcoming
  article~\cite{so-called-paper-2}. 
\end{remark}

\section{Approximation of the generating function of a simplicial affine cone}\label{section-simplicial-cone}

 Let $\c\subset V $ be a simplicial cone with
integral generators $v_j$, $j=1,\ldots,d$, and let $s\in V_\Q$. Let
$k_0 \leq d $. In this section we will obtain an expression for  the
$k_0 +1$ lowest degree homogeneous terms of the meromorphic function
$S(s+\c)(\xi)$. Recall that if $\c$ is unimodular, the function
$S(s+\c)(\xi)$ has a ``short'' expression
$$
S(s+\c)(\xi)=\e^{\la \xi,\bar{s}\ra}\prod_{j=1}^d\frac{1}{1- \e^{\la \xi,v_j\ra}},
$$
where  $v_i$, $i=1,\ldots,d$  are the primitive integral generators
of the edges and $\bar{s}$ is the unique lattice point in the
corresponding parallelepiped $s+\b$. This is a particular case of Lemma
\ref{lemma-cell}.

When $\c$ is not unimodular, it is not possible to compute
efficiently the first $k_0$ terms of the Laurent expansion of the
function $S(s+\c)(\xi)$, if $k_0$  is part of the input as well as
the dimension $d$. In contrast, \emph{if  $k_0$ is fixed}, we are
going to obtain an expression for the terms of degree $ \leq -d+k_0$
which only involves a discrete summation over cones in dimension $\leq k_0$ and determinants. For example,  the lowest degree term is $
\mathopen|\det(v_j)|\prod_j \frac{-1}{\langle \xi,v_j\rangle}$.

\subsection{Patching functions} \label{patchingfun}

For constructing the approximation, we will use  a \emph{patching
function}. For $I\subseteq \{1,\ldots, d\}$, we denote by $L_I$ the
linear span of the vectors $v_i$, $i\in I$ and by
$L_I^\perp\subseteq V^* $ the orthogonal subspace.  We denote by $I^c$ the complement of
$I$ in $\{1,\ldots ,d\}$. 
\begin{definition}
We denote by  $\Jposet{d}{d_0}$ the set of subsets $I\subseteq
\{1,\dots,d\}$ of cardinality $|I|\geq d_0$. A   function $I\mapsto
\lambda(I)$ on $\Jposet{d}{d_0}$ is called a \emph{patching  function} if it
satisfies the following condition.
\begin{equation}\label{Barvinok-patchfunction}
\charfun[bigg]{\bigcup_{I\in \Jposet{d}{d_0}}L_I^\perp} = \sum_{I\in
\Jposet{d}{d_0}}\lambda(I)\charfun[big]{L_{I}^\perp}.
\end{equation}
\end{definition}
\begin{remark}
  The family of subspaces $L_I$, $|I|\geq d_0$  is closed under sum, and
  the family of orthogonals~$L_I^\perp$ is closed under intersection. 
  The value $\lambda(I)$ plays the same as role as the M\"obius function
  $\mu(L)$ for $L_I^\perp$ that Barvinok \cite[section 7]{barvinok-2006-ehrhart-quasipolynomial} computes
  algorithmically for a certain family of subspaces $L$ by walking the poset.    
  From this discussion, it follows that patching functions do exist.  
  The precise relation between Barvinok's construction and the construction of
  the present paper will be studied in the forthcoming
  paper~\cite{so-called-paper-3}. 
\end{remark}
We will compute a canonical patching function below,
in Proposition~\ref{rho}.   

Let us state some interesting properties.
\begin{lemma}\label{patching_function}
Let $I\mapsto \lambda(I)$ be a function on $\Jposet{d}{d_0}$. The
following conditions are equivalent.
\begin{enumerate}[\rm(i)]
\item  $\lambda$ is a patching function.

\item $ \sum_{I\in \Jposet{d}{d_0}, I\subseteq I_0}\lambda(I)=1 $
  for every $ I_0\in \Jposet{d}{d_0}$.

\item For $1\leq i\leq d$, let $F_i(z)\in\C[[z]]$ be a formal power series
  (in one variable) with constant term equal to $1$. Then
  \begin{multline}\label{product-power-series}
    \prod_{1\leq i\leq d}F_i(z_i)\equiv \sum_{I\in
      \Jposet{d}{d_0}}\lambda(I)\prod_{i\in I^c
    }F_i(z_i) \\
    \mbox { mod terms of } z\mbox {-degree } \geq d-d_0+1.
  \end{multline}
\item Let $z_{I^c}= \sum_{i\in I^c}z_i$. Then
\begin{multline}
\e^{z_1+\cdots + z_d} \equiv \sum_{I\in \Jposet{d}{d_0}}\lambda(I)
\e^{z_{I^c}} 
\mbox { mod terms of } z\mbox {-degree } \geq d-d_0+1.
\end{multline}
\end{enumerate}
\end{lemma}
\begin{proof}
Let $I_0\in \Jposet{d}{d_0}$. Then  there exists  $\xi\in L_{I_0}^\perp$
such that $\xi\in L_{I}^\perp$ if and only if $L_{I_0}^\perp \subseteq
L_I^\perp$, i.e., if and only if $I\subseteq I_0$. Thus (i)
$\Leftrightarrow$ (ii).

Let us prove that (ii) $\Rightarrow$ (iii). We write $F_i(z_i)= 1+
z_ig_i(z_i)$. We have
\begin{equation}\label{product-expansion}
\prod_{1\leq i\leq d}(1+z_ig_i(z_i))=\sum_{K\subseteq
\{1,\dots,d\}}\prod_{i\in K}z_ig_i(z_i).
\end{equation}
Consider a monomial $z_1^{k_1}\cdots z_d^{k_d}$ of total degree
$k_1+ \cdots + k_d \leq d-d_0$.  Let us denote its coefficient
 in the product $\prod_{i\in
K}z_ig_i(z_i)$ by $\alpha_K$. Let $K_0$ be the set of indices such
that $k_i\neq 0$. Then $|K_0|\leq d-d_0$.  Moreover in the right
hand side of (\ref{product-expansion}), our monomial appears only in
the terms where $K\subseteq K_0$. Therefore the coefficient of
$z_1^{k_1}\cdots z_d^{k_d}$ in $\prod_{1\leq i\leq d}F_i(z_i)$ is
equal to $\sum_{K\subseteq K_0}\alpha_K$. Furthermore, the
coefficient of $z_1^{k_1}\cdots z_d^{k_d}$  in the right-hand-side
of (\ref{product-power-series}) is equal to
$$
\sum_{I\in \Jposet{d}{d_0}} \lambda(I)\sum_{K\subseteq K_0\cap
I^c}\alpha_K = \sum_{K\subseteq K_0}\alpha_K \sum_{\substack{I\in \Jposet{d}{d_0}\\
K\subseteq I^c}}\lambda(I).
$$
By condition (ii) we have $$\sum_{\substack{I\in \Jposet{d}{d_0}\\ K\subseteq
I^c}}\lambda(I)=1\quad\text{for every $K\subseteq K_0$}.$$
Thus we have proved
that (ii) $\Rightarrow$ (iii). Next, (iv) is a particular case of
(iii), so it remains only to prove that (iv) implies (ii).

By expanding the exponentials in  condition (iv), we obtain
$$
(z_1+\dots + z_d) ^{d-d_0}=  \sum_{I\in \Jposet{d}{d_0}}\;\;\lambda(I)
\Bigl(\sum_{i\in I^c}z_i \Bigr)^{d-d_0}.
$$
Condition  (ii) follows easily from this relation.
\end{proof}

\subsection{Formula for intermediate sums}

In preparation for the approximation theorem, we need some notations and an expression
for intermediate sums~$S^L(s+\c,\Lambda)(\xi)$.

 We have $V=L_I\oplus L_{I^c}$. For $x\in
V$ we denote the components
 by
$$
x=x_I + x_{I^c}.
$$
Thus we identify the quotient $V/L_{I}$ with $L_{I^c}$ and   we
denote the projected lattice by $\lattice_{I^c}\subset L_{I^c}$.
Note that $L_{I^c} \cap \lattice \subseteq \lattice_{I^c}$, but the
inclusion is strict in general.
\begin{example}
Let $v_1=(1,0), v_2=(1,2)$, $I=\{2\}$. The projected lattice~$\Lambda_{I^c}$ on $L_{I^c}=\R v_1$
is $\Z \frac{v_1}{2}$.  See Figure~\ref{fig:projected-lattice}.
\begin{figure}[t]
  \centering
  \ifpdf
  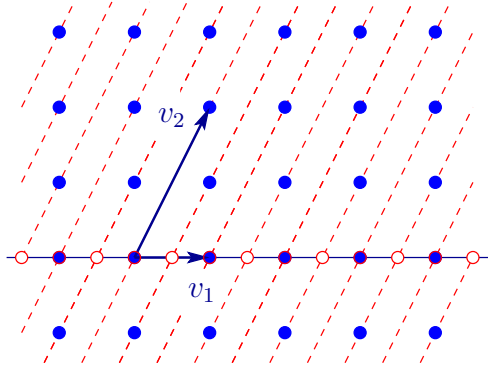
   \else
   \input{projected-lattice.pstex_t}
   \fi
  \caption{The projected lattice~$\Lambda_{\{1\}}$}
  \label{fig:projected-lattice}
\end{figure}
\end{example}
 We denote by   $\c_{I}$ the cone  generated by the vectors  $v_j$, for $j\in
I$ and by $\b_I$  the parallelepiped
 $\b_I=\sum_{i\in I}[0,1\mathclose[\, v_i$.
   Similarly
 we denote by $\c_{I^c}$ the cone generated by  the vectors  $v_j$, for $j\in
I^c$ and $\b_{I^c}=\sum_{i\in I^c}[0,1\mathclose[\, v_i$.
The projection of the cone $\c$ on $V/L_I=L_{I^c}$ identifies with $\c_{I^c}$. Note that the generators
$v_i$, $i\in I^c,$ may be non-primitive for the projected lattice $\lattice_{I^c}$,
even if it is primitive for $\lattice$, as we see in the previous
example.
We write $s=s_I+s_{I^c}$. 

We first show that the intermediate generating function~$S^{L_I}(s+\c,\Lambda)$
decomposes as a product. 

\medskip

The function $S(s_{I^c}+\c_{I^c},\lattice_{I^c})(\xi)$ is a meromorphic function
on the space $(L_{I^c})^*$. The integral
$I(s_I+\c_I,L_I\cap \lattice)(\xi)$ is a meromorphic function
on the space $(L_I)^*$. We consider both as functions on $V^*$
through the decomposition $V=L_I\oplus L_{I^c}$.
\begin{proposition}  \label{petite-somme}
The intermediate sum for the full cone $s+\c$ breaks up into the
product
\begin{equation}\label{petite-somme-eqn}
 S^{L_I}(s+\c,\Lambda)(\xi) =
 S(s_{I^c}+\c_{I^c},\lattice_{I^c})(\xi)\, I(s_I+\c_I,L_I\cap
 \lattice)(\xi).
\end{equation}
\end{proposition}
\begin{proof}
The projection of the cone $s + \c$ into~$L_{I^c}$ is the cone $s_{I^c}+\c_{I^c}$.  For
each $x_{I^c}\in (s_{I^c}+\c_{I^c})\cap \lattice_{I^c}$, the slice $(s+\c ) \cap
(x_{I^c}+ L_I)$  is the cone $x_{I^c} +s_I+\c_I$. Let us compute
the integral on the slice.
\begin{equation}\label{slice}
  \int_{(s+\c )\cap (x_{I^c}+ L_I)}\e^{\la \xi,y \ra} \,\mathrm dm_{L_I\cap\Lambda}(y).
\end{equation}
We write $y =x_{I^c}+ s_I+ \sum_{j\in I} y_j v_j$. 
Then
$$\mathrm dm_{L_I\cap\Lambda}(y) = \vol_{L_I\cap \lattice}(\b_I)\prod_{j\in
  I}\mathrm dy_j.$$
Hence (\ref{slice})  is equal to
$$
\e^{\la \xi,x_{I^c}\ra}\e^{\la \xi,s_I \ra}\vol_{L_I\cap
\lattice}(\b_{I})(-1)^{|I|}\prod_{j\in I}\frac{1}{\la
\xi,v_j\ra}.
$$
We observe that only the
first factor, $\e^{\la \xi,x_{I^c}\ra}$, depends on~$x_{I^c}$.  The sum of these
factors over all $x_{I^c}\in (s_{I^c}+\c_{I^c})\cap \lattice_{I^c}$ gives
$S(s_{I^c}+\c_{I^c},\lattice_{I^c})(\xi)$, and using formula~\eqref{I} for
the integral, we obtain~\eqref{petite-somme-eqn}.
\end{proof}

\subsection{Approximation theorem}

We can now state and prove the approximation theorem. 

\begin{theorem}[Approximation by a patched generating function]\label{approx-generating-function}
Let $\c\subset V$ be a rational simplicial cone with  edge
generators $v_1,\dots,v_d$. Let $s\in V_\Q$. Let $I\mapsto
\lambda(I)$ be a patching function on $\Jposet{d}{d_0}$. For $I\in
\Jposet{d}{d_0}$ let $L_I$ be the linear span of $\{v_i\}_{ i\in I}$. Then
we have
\begin{multline}\label{belle-formule-box}
  S(s+\c,\Lambda)(\xi)\equiv  A^\lambda(s+\c,\Lambda)(\xi) 
  := \sum_{I\in \Jposet{d}{d_0}}\lambda(I)\, S^{L_I}(s+\c, \Lambda)(\xi)
\\ \mbox{ mod terms of  }  \xi\mbox{-degree } \geq -d_0+1.
\end{multline}
\end{theorem}
We call the function~$A^\lambda(s+\c,\Lambda)(\xi)$ on the right-hand side of~\eqref{belle-formule-box} the
\emph{patched generating function} of~$s+\c$ (with respect to~$\lambda$). 

\begin{proof}[Proof of Theorem~\ref{approx-generating-function}]
We write the vertex as $s=\sum_{i} s_i v_i$. Let $a=\sum_{i} a_i
v_i\in V$. We apply (\ref{product-power-series}) to the functions
$$
F_i(z_i)=\e^{(a_i-s_i) z_i}\frac{-z_i}{1-\e^{z_i}},
$$
and we substitute  $z_i=\langle \xi,v_i\rangle$. We obtain
\begin{multline*}
\e^{\langle \xi, a-s\rangle}\prod_{i=1}^d \frac{-\langle
\xi,v_i\rangle}{1-\e^{\langle \xi,v_i\rangle}}\equiv \sum_{I\in
\Jposet{d}{d_0}}\lambda(I)\e^{\langle
\xi,a_{I^c}-s_{I^c}\rangle}\prod_{i\in I^c}\frac{-\langle
\xi,v_i\rangle}{1-\e^{\langle \xi,v_i\rangle}}\\ \mbox{ mod terms of
}  \xi\mbox{-degree } \geq d-d_0+1.
\end{multline*}
We multiply both sides first by $\e^{\langle \xi, s\rangle}$; because this is
analytic in~$\xi$ and thus of non-negative~$\xi$-degree, the identity modulo
terms of high~$\xi$-degree still holds true.  Then we multiply by $1/\prod_{i=1}^d
(-\langle \xi,v_i\rangle)$, which is homogeneous of degree~$-d$ in~$\xi$. We obtain
\begin{multline}\label{un-point-dans-la-boite}
\e^{\langle \xi, a\rangle}\prod_{i=1}^d \frac{1}{1-\e^{\langle
\xi,v_i\rangle}}\equiv \sum_{I\in \Jposet{d}{d_0}}\lambda(I)\e^{\langle
\xi,a_{I^c}\rangle}\prod_{i\in I^c}\frac{1}{1-\e^{\langle
\xi,v_i\rangle}} \e^{\langle \xi,
s_I\rangle}\prod_{i\in I} \frac{-1}{\langle \xi,v_i\rangle}\\
\mbox{ mod terms of }  \xi\mbox{-degree } \geq -d_0+1.
\end{multline}
Now, we sum up  equalities (\ref{un-point-dans-la-boite}) when $a$
runs over the set $(s+\b)\cap \lattice$ of  integral points in the
fundamental parallelepiped $s+\b$ of the affine cone $s+\c$. On the left-hand side we obtain
$$
\sum_{a\in (s+\b)\cap \lattice}\e^{\langle \xi, a\rangle}\prod_{i=1}^d
\frac{1}{1-\e^{\langle \xi,v_i\rangle}}.
$$
By Lemma \ref{lemma-cell-generating_function}, this is precisely
$S(s+\c,\Lambda)(\xi)$. On the right-hand side, for each $I$, we have a sum over $a\in
(s+\b)\cap \lattice$ of the function
$$\e^{\langle
\xi,a_{I^c}\rangle}\prod_{i\in I^c}\frac{1}{1-\e^{\langle
\xi,v_i\rangle}} \e^{\langle \xi,
s_I\rangle}\prod_{i\in I} \frac{-1}{\langle \xi,v_i\rangle},$$
which depends only on the
projection $a_{I^c}$ of $a$ in the decomposition $a=a_I+a_{I^c}\in
L_I\oplus L_{I^c}$. When $a$ runs over $(s+\b)\cap \lattice$, its
projection $a_{I^c}$ runs over $(s_{I^c}+\b_{I^c})\cap
\lattice_{I^c}$. Let us show that the fibers have the same number of
points, equal to $\vol_{L_I\cap\Lambda}(\b_I)$. For a given $a_{I^c}\in
(s_{I^c}+\b_{I^c})\cap \lattice_{I^c}$, let us compute the fiber
$$ \{\,y\in (s+\b)\,\cap \lattice :  y_{I^c}=a_{I^c}\,\}.
$$
Fix  a point $a_I+ a_{I^c}$ in this fiber. Then $y=a_{I^c}+y_I$ lies
in the fiber if and only if $y_{I}-a_{I}\in (s_{I}-
a_{I}+\b_{I})\,\cap \lattice$. By Lemma \ref{lemma-cell}(ii), the
cardinality of the fiber is equal to $\vol_{L_I\cap\Lambda}(\b_{I})$. Thus we obtain
\begin{multline}\label{approx-cone-last}
S(s+\c)(\xi)\equiv \\
\sum_{I\in
\Jposet{d}{d_0}}\lambda(I)S(s_{I^c}+\b_{I^c})(\xi)
  \prod_{i\in I^c}\frac{1}{1-\e^{\langle
\xi,v_i\rangle}} \e^{\langle \xi,
s_I\rangle}\vol_{L_I\cap\Lambda}(\b_{I})\prod_{i\in I} \frac{-1}{\langle \xi,v_i\rangle}\\
\mbox{ mod terms of }  \xi\mbox{-degree } \geq -d_0+1.
\end{multline}
By Proposition~\ref{petite-somme} and Lemmas
\ref{lemma-integral} and~\ref{lemma-cell-generating_function},  the term
corresponding to an $I\in\Jposet{d}{d_0}$ in the
right-hand side of (\ref{approx-cone-last}) is precisely
$$
S^{L_I}(s+\c,\Lambda)(\xi) = S(s_{I^c}+\c_{I^c},\lattice_{I^c})(\xi)\,I(s_I+\c_I,L_I\cap \lattice)(\xi),
$$
which completes the proof.
\end{proof}

\begin{remark}
  For $d_0=0$, we obtain the poset $\Jposet{d}{0}$ of all subsets
  of~$\{1,\dots,d\}$. The unique patching function on $\Jposet{d}{0}$ is given
  by 
  $\lambda(\emptyset)=1$ and $\lambda(I)=0$ for all~$I\neq\emptyset$.  Then
  the approximation is trivial, i.e., 
  $S(s+\c,\Lambda)(\xi) = A^\lambda(s+\c,\Lambda)(\xi)$.
\end{remark}

{\small
\begin{example}
Let $\c$ be the standard cone in $\R^2$, and $d_0=1$. Thus
$\Jposet{2}{1}$ consists of three subsets, $\{1\}, \{2\}$ and $\{1,2\}$.
A patching function is given by $\lambda(\{i\})=1$ and
$\lambda(\{1,2\})=-1$. We consider the affine cone $s+\c$ with $s=(-\frac12,-\frac12)$.
Let $\xi=(\xi_1,\xi_2)$. We have
\begin{align*}
  I(s_i+\c_{\{i\}})(\xi) &= \frac{-\e^{-\xi_i/2}}{\xi_i}, &&&
  I(s+\c)(\xi)&=\frac{\e^{-\xi_1/2-\xi_2/2}}{\xi_1\xi_2},\\
  S(s_i+\c_{\{i\}})(\xi) &= \frac{1}{1-\e^{\xi_i}},  &&&
  S(s+\c)(\xi)&=\frac{1}{(1-\e^{\xi_1})(1-\e^{\xi_2})}.
\end{align*}
 The approximation theorem claims that
\begin{multline*}
\frac{1}{(1-\e^{\xi_1})(1-\e^{\xi_2})}\equiv
\frac{1}{1-\e^{\xi_2}}\cdot \frac{-\e^{-\xi_1/2}}{\xi_1}+
\frac{1}{1-\e^{\xi_1}}\cdot \frac{-\e^{-\xi_2/2}}{\xi_2}-
\frac{\e^{-\xi_1/2-\xi_2/2}}{\xi_1\xi_2}
\\\mbox{ mod terms of } \xi\mbox{-degree } \geq 0.
\end{multline*}
Indeed, the difference between the two sides is equal to
$$
\Bigl(\frac{1}{1-\e^{\xi_1}}+\frac{\e^{-\xi_1/2}}{\xi_1}\Bigr)\Bigl(\frac{1}{1-\e^{\xi_2}}+\frac{\e^{-{\xi_2/2}}}{\xi_2}\Bigr)
$$
which is analytic near $0$.
\end{example}}

\subsection{An explicit patching function}

Next we compute an explicit patching function on $\Jposet{d}{d_0}$. It is
related to the M\"obius function of the poset $\Jposet{d}{d_0}$, so we call
it the \emph{M\"obius patching function} and denote it by
$\lambda_{\Moebius}$.  We will denote the corresponding patched generating
function~$A^{\lambda_{\Moebius}}(s+\c,\Lambda)$ 
%% with respect to the canonical patching function~$\lambda_{\Moebius}$ on
%% $\Jposet{d}{d_0}$
by~$A_{\geq d_0}(s+\c,\Lambda)$.
%%  We denote the
%% binomial coefficient $\frac{p!}{k! (p-k)!}$ by $\binom{p}{k}$.  
%%%%% omitted, standard notation --Matthias
\begin{proposition}\label{rho} For $I\in \Jposet{d}{d_0}$, let
$$
\lambda_{\Moebius}(I)= (-1)^{|I|-d_0}\binom{|I|-1}{d_0-1}.
$$
Then $\lambda_{\Moebius}$ is a patching function on $\Jposet{d}{d_0}$.
\end{proposition}
\begin{proof}
We prove that $\lambda_{\Moebius}$ satisfies Condition (iv) of Lemma
\ref{patching_function}.
 The trick is to write
$\e^z= 1+t( \e^z -1)|_{t=1}$. Thus
$$
  \e^{z_1+\dots + z_d}=\prod_1^d \e^{z_i} =\prod_{i=1}^d \bigl(1+t( \e^{z_i} -1)\bigr)\Big|_{t=1}
$$
Let us consider $P(t): = \prod_{i=1}^d \bigl(1+t( \e^{z_i} -1)\bigr)= \sum_{q=0}^d
C_q(z)t^q$ as a polynomial in the indeterminate $t$. As $ \e^{z_i}
-1$ is a sum of terms of $z_i$-degree $>0$, we have
\begin{equation}\label{exponential}
 \e^{z_1+\dots + z_d}\equiv \sum_{q=0}^{k_0} C_q(z) \mbox{ mod terms of  } z\mbox{-degree } \geq k_0+1.
\end{equation}
Next, we write
$$
P(t)= \prod_1^d \bigl(1+t( \e^{z_i} -1)\bigr)= \prod_1^d \bigl((1-t) + t \e^{z_i}\bigr).
$$
By expanding the product, we obtain
$$
C_q(z)= \sum_{ |K|\leq q}(-1)^{q-|K|}\binom{d-|K|}{q-|K|}\e^{z_K}.
$$
Summing up these coefficients for $0\leq q\leq k_0=d-d_0$, we obtain
$$
\sum_{q=0}^{k_0} C_q(z)=\sum_{ |K|\leq k_0}\left(\sum_{q=|K|}^{k_0}
(-1)^{q-|K|}\binom{d-|K|}{q-|K|}\right )\e^{z_K}.
$$
By substituting $K=I^c$ and $d-q=m$, we obtain
$$
\sum_{q=0}^{k_0} C_q(z)= \sum_{|I|\geq d_0}f(|I|)\e^{z_{I^c}},
$$
with
$$
f(j)=\sum_{m=d_0}^{j}(-1)^{j-m}\binom{j}{j-m}=\sum_{m=d_0}^{j}(-1)^{j-m}\binom{j}{m}.
$$
The truncated binomial sum $f(j)$ is easy to compute, using the
recursion relation $\binom{j}{m}= \binom{j-1}{m-1}+ \binom{j-1}{m}$.
We obtain
$$
f(j)= (-1)^{j-d_0}\binom{j-1}{d_0-1}.
$$
Thus, $\lambda_{\Moebius}(I)=f(|I|)$ satisfies Condition (ii) for a
patching function.
\end{proof}
\begin{remark}
Proposition \ref{rho} can also be deduced from results in
\cite{Bjorner-Lovasz-Yao}.
\end{remark}

\section{Computation of the patched generating function
  %of a simplicial cone
} 
\label{section-patchedgenfun-computation}

In this section, we show that if~$k_0 = d - d_0$ is fixed, the patched
generating function~$A_{\geq d_0}(s+\c,\Lambda)$ can be efficiently computed for
a simplicial cone~$s+\c$.  This will be a consequence of Barvinok's
polynomial-time decomposition of cones in fixed dimension \cite{bar,barvinokzurichbook}.  We exhibit the
dependence of the patched generating function on the vertex~$s$ explicitly
as a ``step function'' in two useful ways, using the ``ceiling'' function $\lceil\cdot\rceil$ and the
``fractional part'' function~$\{\cdot\}$, respectively. 

\smallbreak

We will write the patched generating function using the following analytic
function. 
\begin{definition}
  Let
  \begin{equation}\label{bernoulli-genfun}
    T(\tau ,x)=\e^{\tau x  }\frac{x}{1-\e^x} = -\sum_{n=0}^\infty B_n(\tau ) \frac{x^n}{n!},
  \end{equation}
  where $B_n(\tau)$ are the Bernoulli polynomials.
\end{definition}

We start with the following result.
\begin{theorem}[Short formula for~$S^{L_I}(s+\c,\Z^d)(\xi)$ %%
  %% , where $L$ is parallel to a face of the cone~$\c$
  for varying~$s$%
  ]\label{th:problem1}
Fix a non-negative  integer $\co$.
There exists a polynomial time algorithm for the following problem.
Given the following {input}:
\begin{inputlist}
\item a number $d$ in unary encoding,

\item a simplicial cone $\c = \c(v_1,\dots,v_d) \subset \R^d$,
  represented by the primitive vectors $v_1,\dots,v_d\in\Z^d$ in binary
   encoding,

 \item a subspace $L_I= \lin(v_i : i\in I)\subseteq \R^d$ of codimension $\co $, represented by 
   an index set $I \subseteq \{1,\dots,d\}$ of cardinality $d_0 = d-\co$,
\end{inputlist}
compute the following {output} in binary encoding:
\begin{outputlist}
\item a finite set~$\Gamma$,
\item for every $\gamma$ in $\Gamma$, integers~$\alpha^{(\gamma)}$,
  rational vectors~$\eta ^{(\gamma)}_i$ 
  and $w^{(\gamma)}_i$ for $i=1,\dots,d$, where $\eta ^{(\gamma)}_i\in \Z^d$ for $i\in I^c$
\end{outputlist}
such that for every $s\in\Q^d$, we have the following equality of meromorphic
functions of~$\xi$:
\begin{subequations}\label{problem1-shortformula-both}
  \begin{align}
    &S^{L_I}(s+\c,\Z^d)(\xi)
    %% = S(s_{I^c}+\c_{I^c})(\xi)\,I(s_I+\c_I)(\xi)\\
    \notag\\
    &\quad = \sum_{\gamma\in \Gamma} \alpha^{(\gamma)} \prod_{i\in I^c}
    T\bigl(\bigl\lceil\bigl\langle \eta _i^{(\gamma)}, s\bigr\rangle\bigr\rceil,
    \bigl\langle
    \xi,w^{(\gamma)}_i\bigr\rangle\bigr) \notag\\
    &\qquad\qquad \cdot \prod_{i\in I} \exp\bigl({\bigl\langle \eta _i^{(\gamma)},
      s\bigr\rangle \bigl\langle \xi, w_i^{(\gamma)}\bigr\rangle}\bigr)
    \cdot\frac{1}
    {\prod_{i=1}^d\bigl\langle \xi,w^{(\gamma)}_i\bigr\rangle} \label{problem1-shortformula}\\
    &\quad= \e^{\langle \xi,s\rangle} \sum_{\gamma\in \Gamma} \alpha^{(\gamma)} 
    \prod_{i\in I^c} T\bigl(\bigl\{ -\bigl\langle \eta _i^{(\gamma)},
    s\bigr\rangle\bigr\}, \bigl\langle \xi,w^{(\gamma)}_i\bigr\rangle \bigr) \cdot\frac{1} {\prod_{i=1}^d\bigl\langle
      \xi,w^{(\gamma)}_i\bigr\rangle}. \label{problem1-shortformula-variant}
  \end{align}
\end{subequations}
\end{theorem}

Of course, for $I=\emptyset$ we have $L=\{0\}$, and so we recover formulas
for $S(s+\c,\Z^d)(\xi)$.  If we set $I=\{1,\dots,d\}$, then $L=\R^d$, and we
get formulas for $I(s+\c,\Z^d)(\xi)$.

\begin{remark}
Consider the term corresponding to $\gamma\in\Gamma$
in \eqref{problem1-shortformula} or \eqref{problem1-shortformula-variant}.
As it will follow from the proof, the vector $w_i^{(\gamma)}$ for $i\in I$ is
just the original vector~$v_i$, and the collection $w_i^{(\gamma)}$,
$i=1,\dots,d$,  forms a basis of~$\R^d$. Furthermore  the vectors
$w_i^{(\gamma)}$, with $i\in I^c$, are in $L_{I^c}$ and   form  a basis of the
projected lattice. 
The vectors $\eta _i^{(\gamma)}$, $i=1,\dots,d$,  are the dual (biorthogonal)
vectors to the elements $w_j^{(\gamma)}$,  $j=1,\dots,d$, i.e., $\bigl\langle
\eta _i^{(\gamma)}, w_j^{(\gamma)}\bigr\rangle = \delta_{i,j}$.  
Thus we only need to compute the integers $\alpha^{(\gamma)}$ and the elements
$w_i^{(\gamma)}$ where $i\in I^c$. 
\end{remark}
\begin{remark}
Consider the term corresponding to $\gamma\in\Gamma$ in \eqref{problem1-shortformula-variant}.
As the vectors $w_i^{(\gamma)}$, $i\in I^c$, form a basis of the projected
lattice, we may identify $V/(L_I+\lattice)$ to $\bigoplus_{i\in I^c}
[0,1\mathclose[ \,w_i^{(\gamma)}$.  
Define $$s^{(\gamma)}=\sum_{i\in I^c}\bigl\{-\bigl\langle \eta _i^{(\gamma)},s\bigr\rangle\bigr\}\, w_i^{(\gamma)}.$$
As the $\eta _i^{(\gamma)}$ for $i\in I^c$ are integer vectors, and
$\bigl\langle \eta _i^{(\gamma)},v_j\bigr\rangle=0$ if $j\in I$,
we can think of $s\mapsto s^{(\gamma)}$ as a  linear  
map on  the torus $V/(\Lambda+L_I)$  with integer coefficients.
The point $s+s^{(\gamma)}$ is in $\bigoplus_{i\in I^c} \Z w_i^{(\gamma)} \oplus
\bigoplus_{i\in I} \R v_i$, and  formula~\eqref{problem1-shortformula-variant} reads also 
\begin{equation}
  S^{L_I}(s+\c,\Z^d)(\xi)
=\sum_{\gamma\in \Gamma} \alpha^{(\gamma)} \e^{\langle \xi,s+s^{(\gamma)}\rangle}
     \frac{1} {\prod_{i\in I^c}\bigl(1-\e^{\langle \xi, w_i^{(\gamma)}\rangle}\bigr)}
    \frac{1}{\prod_{i\in I} \langle \xi, v_i\rangle 
      \vphantom{\e^{\langle \xi, w_i^{(\gamma)}\rangle}} %%% vphantom only!
    }.
  \end{equation}
\end{remark}

Now we prove the theorem.

\begin{proof}[Proof of Theorem~\ref{th:problem1}]
Let us describe the algorithm along the proof.  Let $\Lambda=\Z^d$.
By Proposition~\ref{petite-somme}, 
\begin{equation}
  S^{L_I}(s+\c,\Lambda)(\xi) =
  S(s_{I^c}+\c_{I^c},\lattice_{I^c})(\xi)\, I(s_I+\c_I,L_I\cap \lattice)(\xi).
\end{equation}
We first discuss $I(s_I+\c_I, L_I\cap \lattice)$. We have
\begin{equation}
  I(s_I+\c_I, L_I\cap \lattice)(\xi) = \e^{\la \xi,s_{I} \ra}\vol_{L_I\cap \lattice}(\b_{I}) \prod_{j\in I} \frac{-1}{\la \xi,v_j\ra}.\label{coeff-2}
\end{equation}
Using linear functionals~$\eta _i\in \Q^d$, $i\in I$ (the coordinate functions with respect
to the basis $v_i$), write $s_I = \sum_{i\in I}
\langle \eta _i, s\rangle v_i$.  The $\eta _i$ can be read off in polynomial time from
the inverse of the matrix whose columns are~$v_1,\dots,v_d$. Then $\e^{\la
  \xi,s_{I} \ra}$ takes the form 
\begin{equation}
  \e^{\la \xi,s_{I} \ra} = \prod_{i\in I} \exp(\langle \eta _i, s\rangle \langle\xi, v_i\rangle ).
\end{equation}

Now, to handle the factor $S(s_{I^c}+\c_{I^c}, \lattice_{I^c})$, 
note that $\c_{I^c}\subset L_{I^c}$ is a $\co$-dimensional cone. 
By using a Hermite normal form computation, which is polynomial time
\cite{kannan-bachem:79}, we can compute a linear change of variables which
replaces the projected lattice~$\Lambda_{I^c}$ on $L_{I^c}$ by $\Z^{\co}$.
Then, using Barvinok's decomposition \cite{bar}, we decompose it into a  family of  cones
which are {unimodular},
\begin{equation}\label{coeff-3}
  \charfun[big]{\c_{I^c}} \equiv \sum_{m\in M} \epsilon_m \charfun[big]{\c^{(m)}_{I^c}}
  \quad\text{(modulo cones containing lines)},
\end{equation}
where $\epsilon_m\in\{\pm1\}$.
As $\co$ is fixed, this decomposition can be done by a polynomial time
algorithm. This step  is of course  crucial  with respect to the
efficiency of the whole algorithm.

Changing again notations, we now denote by $\c=\c(\{w_i\}_{i\in I^c})$ one of
these unimodular cones~$\smash[t]{\c^{(m)}_{I^c}}\subset L_{I^c}$, with
primitive generators $w_i$, and also write $\epsilon=\epsilon_m$. 
We remark that the vectors $w_i$, $i\in I^c$, generate the projected lattice
on~$L_{I^c}$.  Using linear functionals~$\eta _i\in \Q^d$, $i\in I^c$, write
$s_{I^c} = \sum_{i\in I^c} \langle \eta _i, s\rangle w_i$. 
Actually, we have $\eta _i\in\Z^d$.  By letting $w_i = v_i$ for the other indices
$i\in I$, we can write
\begin{equation}
  \label{eq:s-reunited}
  s = \sum_{i\in I} \langle \eta _i, s\rangle v_i +  \sum_{i\in I^c} \langle \eta _i,
  s\rangle w_i 
  = \sum_{i=1}^d \langle \eta _i, s\rangle w_i.
\end{equation}

Let $s'_{I^c}$ be the unique lattice point in the fundamental parallelepiped of the
cone $s_{I^c}+\c$.  We have
\begin{equation}
  s'_{I^c} = \sum_{i\in I^c} \lceil\langle \eta _i, s\rangle\rceil w_i.
\end{equation}
Using this, we obtain the generating function from
Lemma~\ref{lemma-cell-generating_function} as
$$
S(s_{I^c}+\c)(\xi)=\frac{\e^{\langle \xi,s'_{I^c}\rangle}}{\prod_{i\in I^c} (1-\e^{\langle \xi,w_i\rangle})}.
$$
Thus finally, using \eqref{coeff-2} we have the meromorphic function
\begin{equation}\label{unimodular}
\epsilon \vol_{L_I\cap \lattice}(\b_{I}) (-1)^{|I|}
\cdot \frac{\e^{\langle \xi,s'_{I^c}\rangle}}{\prod_{i\in I^c} (1-\e^{\langle \xi,w_i\rangle})}\cdot\frac{\e^{\la \xi,s_I \ra}}{\prod_{j\in I}{\la
\xi,v_j\ra}}.
\end{equation}
Then \eqref{unimodular} is now written
\begin{equation}\label{unimodular-2}
  \alpha \prod_{i\in I^c} T(\lceil\langle \eta _i, s\rangle\rceil, \langle
  \xi,w_i\rangle)\cdot
  \prod_{i\in I} \exp(\langle \eta _i, s\rangle \langle\xi, v_i\rangle )
  \cdot \frac{1}
  {\prod_{i=1}^d\langle \xi,w_i\rangle},
\end{equation}
where $\alpha$ collects the multiplicative constants in~\eqref{unimodular}.
Collecting these terms gives the desired short formula~\eqref{problem1-shortformula}.

To derive the second form, we note $\lceil\langle \eta _i, s\rangle\rceil =
\langle \eta _i, s\rangle + \{ -\langle \eta _i, s\rangle \}$, so we can write
\begin{equation}
  T(\lceil\langle \eta _i, s\rangle\rceil, \langle
  \xi,w_i\rangle) = T(\{ -\langle \eta _i, s\rangle\} , \langle
  \xi,w_i\rangle) \exp(\langle \eta _i, s\rangle \langle
  \xi,w_i\rangle).
\end{equation}
Thus the term~\eqref{unimodular-2} can be written as
\begin{equation}\label{unimodular-2-variant}
  \alpha
  \prod_{i\in I^c} T(\{ -\langle \eta _i, s\rangle\} , \langle
  \xi,w_i\rangle) \cdot
  \e^{\langle \xi,s\rangle}
  \cdot \frac{1}
  {\prod_{i=1}^d\langle \xi,w_i\rangle},
\end{equation}
using~\eqref{eq:s-reunited}.  Collecting these terms gives the short
formula~\eqref{problem1-shortformula-variant}. 
\end{proof}

{\small
\begin{example}
Let us give an example of the output of our algorithm, in a small example.
Consider the $3$-dimensional cone with  rays given by the vectors
$(1,1,1)$, $(1,-1,0)$, $(1,1,0)$. This cone is not unimodular.
 We consider the affine cone $s+\c$.
  Our algorithm described in Theorem  \ref{th:problem1} computes any   intermediate generating function $S^L(s+\c,\Z^3)$ when $L$ is a  linear span of a face of $\c$. For $L=\{0\}$ (indexed by the empty set $I$), we obtain the 
meromorphic function $S(s+\c,\Z^3)(\xi)$ (the discrete generating function of the cone $s+\c$). Here $S(s+\c,\Z^3)(\xi)$ depends of $s=(s_1,s_2,s_3)$ and $\xi=(\xi_1,\xi_2,\xi_3)$ and is given by:
  \begin{multline*}
    \exp
    \left( s_{{1}}\xi_{{1}}+s_{{2}}\xi_{{2}}+s_{{3}}\xi_{{3}} \right)\\
    \left( -{\frac {T \left( \left\{ -s_{{3}}+s_{{2}} \right\}
            ,-\xi_{{1}}-\xi_{{2}} \right) T \left( \left\{ -s_{{1}}+s_{{2}}
            \right\} ,\xi_{{1}} \right) T \left( \left\{ -s_{{3}} \right\}
            ,\xi_{{1}}+\xi_{{2}}+\xi_{{3}} \right) }{ \left(
            -\xi_{{1}}-\xi_{{2}} \right) \xi_{{1}} \left( \xi
            _{{1}}+\xi_{{2}}+\xi_{{3}} \right) }} \right. \\
    \left. \qquad\quad{}+{\frac {T \left( \left\{ -s_{{3}}+s_{{2}} \right\}
            ,\xi_{{1}}-\xi_{{2}} \right) T \left( \left\{
              2\,s_{{3}}-s_{{2}}-s_{{1}} \right\} ,\xi_{{1}} \right) T \left(
            \left\{ -s_{{3}} \right\} ,\xi_{{1}}+\xi_{{2}}+\xi_{{3 }} \right)
        }{ \left( \xi_{{1}}-\xi_{{2}} \right) \xi_{{1}} \left(
            \xi_{{1}}+\xi_{{2}}+\xi_{{3}} \right) }} \right).
  \end{multline*}
%% \end{subequations}

If $L=\R v_1$ is the subspace of dimension $1$   generated by the edge $v_1=(1,1,1)$ of the cone $\c$ (so that $L$ is indexed by the subset $I=\{1\}$ of $\{1,2,3\}$), the intermediate generating function
  $S^L(s+\c,\Z^3)$ is given by:

%% \begin{subequations}
%%   \begin{multline}
%%     -{\frac {T \left( \lceil s_{{3}}-s_{{2}} \rceil ,-\xi_{{1}}-\xi_{{2}}
%%         \right) T \left( \lceil s_{{1}}-s_{{2}} \rceil ,\xi_{{1}} \right) \exp
%%         \left( s_{{3}}\xi_{{1}}+s_{{3}}\xi_{{2}}+s_{{3}}\xi_{{3}} \right) }{
%%         \left( -\xi_{{1}}-\xi_{{2}}
%%         \right) \xi_{{1}} \left( -\xi_{{1}}-\xi_{{2}}-\xi_{{3}} \right) }} \\
%%     + {\frac {T \left( \lceil s_{{3}}-s_{{2}} \rceil ,\xi_{{1}}-\xi_{{2}}
%%         \right) T \left( \lceil s_{{1}}+s_{{2}}-2\,s_{{3}} \rceil ,\xi_{{1}}
%%         \right) \exp \left( s_{{3}}\xi_{{1}}+s_{{3}}\xi_{{2}}+s_{{3}}\xi_{{3}}
%%         \right) }{ \left( \xi_{{1}}-\xi_{{2}} \right) \xi_{{1}} \left(
%%           -\xi_{{1}}- \xi_{{2}}-\xi_{{3}} \right) }}.
%%   \end{multline}
%%   Finally, the alternative presentation, using Formula
%%   \ref{problem1-shortformula-variant} instead, is given by
  \begin{multline*}
    \exp \left( s_{{1}}\xi_{{1}}+s_{{2}}\xi_{{2}}+s_{{3}}\xi_{{3}} \right)
    \left( -{\frac { T \left( \left\{ -s_{{3}}+s_{{2}} \right\}
            ,-\xi_{{1}}-\xi_{{2}} \right) T \left(
            \left\{ -s_{{1}}+s_{{2}} \right\} ,\xi_{{1}} \right) }{ \left( -\xi_{{1}}-\xi_{{2}} \right) \xi_{{1}}\left( -\xi_{{1}}-\xi_{{2}}-\xi_{{3}} \right) }} \right. \\
    \left. {}+{\frac {T \left( \left\{ -s_{{3}}+s_{{2}} \right\}
            ,\xi_{{1}}-\xi_{{2}} \right) T \left( \left\{
              2\,s_{{3}}-s_{{2}}-s_{{1}} \right\} ,\xi_{{1}} \right) }{ \left(
            \xi_{{1}}-\xi_{{2}} \right) \xi_{{1}} \left(
            -\xi_{{1}}-\xi_{{2}}-\xi_{{3}} \right) }} \right).
  \end{multline*}
%% \end{subequations}
\end{example}
}

\begin{remark}\label{patchedgenfun-cone-polytime}
  When $k_0=d-d_0$ is fixed, the set $\Jposet{d}{d_0}$ has a polynomially bounded
  cardinality, and it can be enumerated by a straightforward algorithm along
  with the evaluation of the patching function~$\lambda_{\Moebius}$.  Thus we
  can also compute $A_{\geq d_0}(s+\c,\Lambda)(\xi)$ in the same
  form~\eqref{problem1-shortformula} or~\eqref{problem1-shortformula-variant} in polynomial time.
\end{remark}

\section{Computation of Ehrhart quasi-polynomials}\label{section-Ehrhartcomputation} 

We now apply the approximation of the generating functions of the cones at
vertices to the computation of the highest coefficients for a weighted Ehrhart 
quasi-polynomial.  
We first discuss the case when the weight is a power of a linear form. 

\begin{theorem}\label{main-general}
  Let $\p$ be a simple  polytope and let $\CV(\p)$ denote the set of its
  vertices. For each vertex $s\in\CV(\p)$, let $\c_s$ be the tangent cone of~$s$,
  and let $q_s\in \N$ be a positive integer
  %% the smallest    
  %% (We shouldn't insist to use the smallest integer, as it can be 
  %%  nice to obtain an expression where all q_s are the same.)
  such that  $q_s\, s\in \lattice $.
  %% For $n\in\N$, let $n_s$  be the residue of $n$ mod $q_s$, so that
  %% $0\leq n_s\leq q_s-1$.
  Fix a linear form $\ell\in V^*$ and $M$ a nonnegative integer.
  Fix $0\leq k_0\leq d$ and let $d_0=\max\{d-k_0, 0\}$. Then the Ehrhart quasi-polynomial
  $$
  E(\p,\ell,M; n) = \sum_ {x\in n \p\cap \lattice}
  \frac{\langle\ell,x\rangle^M}{M!}\,$$ coincides in degree $\geq
  M+d-k_0$ with the following quasi-polynomial
  \begin{subequations}\label{eq:main-general-both}
    \begin{equation}\label{eq:main-general}
      \sum_{k=0}^{k_0} \sum_{s\in\CV(\p)} (\lfloor n\rfloor_{q_s}) ^{M+d-k}
      \frac{\langle\xi,s\rangle^{M+d-k}}{(M+d-k)!}\, A_{\geq d_0}\bigl(\{n\}_{q_s} s
      + \c_s,\Lambda\bigr)_{[-d+k]}(\xi)
      % \bigg|_{\xi=\ell}
      ,
    \end{equation}
    evaluated at $\xi=\ell$, which can also be written as
    \begin{equation}\label{eq:main-general-variant}
      \sum_{k=0}^{k_0} n^{M+d-k} \sum_{s\in\CV(\p)} 
      \frac{\langle\xi,s\rangle^{M+d-k}}{(M+d-k)!}\, 
      \Bigl( \e^{-\langle\xi,\{n\}_{q_s}s\rangle} A_{\geq d_0}\bigl(\{n\}_{q_s} s +
      \c_s,\Lambda\bigr)(\xi)\Bigr)_{[-d+k]},  %\bigg|_{\xi=\ell}
    \end{equation}
  \end{subequations}
  evaluated at $\xi=\ell$.
\end{theorem}
In the following, we will use the second form~\eqref{eq:main-general-variant}.
  
\begin{remark}
  The sum \eqref{eq:main-general-both} depends polynomially on
$\ell$. However, for an individual vertex $s$, the functions
$$\xi\mapsto \frac{\langle\xi,s\rangle^{M+d-k}}{(M+d-k)!}\, A_{\geq d_0}\bigl(\{n\}_{q_s}
s + \c_s,\Lambda\bigr)_{[-d+k]}(\xi)$$ 
and
$$
\xi\mapsto \frac{\langle\xi,s\rangle^{M+d-k}}{(M+d-k)!}\, \Bigl( \e^{-\langle\xi,\{n\}_{q_s}s\rangle} A_{\geq d_0}\bigl(\{n\}_{q_s} s +
    \c_s,\Lambda\bigr)(\xi)\Bigr)_{[-d+k]} 
$$
are meromorphic functions, which are not
defined  if $\xi$ is singular. Thus in the algorithm we use a
deformation procedure.
\end{remark}
%% \begin{remark}\eqref{eq:main-general} and \eqref{eq:main-general-variant} are
%%   clearly quasi-polynomials in $n$ with 
%%   coefficients which depend only on the residues $\{n\}_{q_s}=(n\bmod q_s),
%%   s\in\CV(\p)$.
%% \end{remark}  
%%%% Removed because I think this is redundant due to the two forms in
%%%% which we show the same function. --Matthias
\begin{proof}[Proof of Theorem~\ref{main-general}]
The sum
$ \sum_ {x\in n \p\cap \lattice} \frac{\langle\xi,x\rangle^M}{M!}$
is the term of $\xi$-degree $M$ in
$$
S(n\p)(\xi)=\sum_{s\in\CV(\p)} S(ns+\c_s)(\xi).
$$
Fix a vertex $s$. We write $n= \lfloor n\rfloor_{q_s} + \{n\}_{q_s}$. As
$\lfloor n\rfloor_{q_s} s$ is a  lattice
point, we have
\begin{equation}\label{dilated-cone}
S(ns+\c_s)(\xi)=\e^{\lfloor n\rfloor_{q_s}\langle\xi, s\rangle}S(\{n\}_{q_s} s+\c_s)(\xi).
\end{equation}
Consider $S(ns+\c_s)_{[M]}{}(\xi)$  as a quasi-polynomial in $n$. By
(\ref{dilated-cone}),  it coincides in degree $\geq M+d-k_0$ with
$$
\sum_{k=0}^{k_0}(\lfloor n\rfloor_{q_s})^{M+d-k}
\frac{\langle\xi,s\rangle^{M+d-k}}{(M+d-k)!}\,S(\{n\}_{q_s}
s+\c_s)_{[-d+k]}(\xi).
$$
Now, for $0\leq k\leq k_0$,  we have
$$
S(\{n\}_{q_s} s+\c_s)_{[-d+k]}(\xi)= A_{\geq d_0}(\{n\}_{q_s} s + \c_s)_{[-d+k]}(\xi).
$$
By specializing on $\xi=\ell$, we obtain the claim in the form of
equation~\eqref{eq:main-general}.\smallbreak

To obtain the second claim in the form of~\eqref{eq:main-general-variant}, 
we write
\begin{equation}\label{dilated-cone-variant}
S(ns+\c_s)(\xi)=\e^{n\langle\xi, s\rangle} \bigl( \e^{-\langle\xi, s\rangle
  \{n\}_{q_s}}S(\{n\}_{q_s} s+\c_s)(\xi) \bigr).
\end{equation}
Again, by expanding we obtain that the quasi-polynomial
$S(ns+\c_s)_{[M]}{}(\xi)$ coincides in degree $\geq M+d-k_0$ with
$$
\sum_{k=0}^{k_0} n^{M+d-k}
\frac{\langle\xi,s\rangle^{M+d-k}}{(M+d-k)!}\,\Bigl(
\e^{-\langle\xi, s\rangle
  \{n\}_{q_s}} S(\{n\}_{q_s}
s+\c_s)(\xi)\Bigr) _{[-d+k]}.
$$
Since $\e^{-\langle\xi, s\rangle \{n\}_{q_s}}$ is analytic in~$\xi$, we have 
for $0\leq k\leq k_0$ that 
\begin{multline}
\Bigl(\e^{-\langle\xi, s\rangle
  \{n\}_{q_s}} S(\{n\}_{q_s} s+\c_s)(\xi)\Bigr)_{[-d+k]}\\
= \Bigl(\e^{-\langle\xi, s\rangle
  \{n\}_{q_s}} A_{\geq d_0}(\{n\}_{q_s} s + \c_s)(\xi)\Bigr)_{[-d+k]}.
\end{multline}

Again, by specializing on $\xi=\ell$, we obtain the claim in the form of
equation~\eqref{eq:main-general-variant}.
\end{proof}

We now derive the coefficients of the weighted Ehrhart polynomial as
short \emph{closed formulas} that are ``step polynomials'' (cf.~\cite{verdoolaege-woods-2005}).  
These can then be evaluated efficiently, providing a corollary (Theorem~\ref{ehrhart-barvinok-style}) in the same form as
Barvinok's theorem in~\cite{barvinok-2006-ehrhart-quasipolynomial}.

\begin{theorem} \label{powersoflinear-stepfun}
 For every fixed number ~$k_0 \in\N$, there exists a polynomial-time algorithm
  for the following problem.

\noindent  Input:
\begin{inputlist}
\item a number~$ d\in\N$ in unary encoding, 
  with $d\geq k_0$, 

\item a finite index set $\CV$,

\item a simple polytope~$\p$, given by its vertices, rational vectors $\ve s_j
  \in \Q^d$ for $j\in\CV$ in binary encoding,
\item a rational vector $\ell \in \Q^d$ in binary encoding,
\item a number~$M \in \N$ in unary encoding.
\end{inputlist}
\noindent Output, in binary encoding,
\begin{outputlist}
\item an index set~$\Gamma$,
\item polynomials $f^{\gamma,m} \in \Q[r_1,\dots,r_{k_0}]$ and integer numbers
  $\zeta^{\gamma,m}_i\in\Z$, $q^{\gamma,m}_i \in \N$ for $\gamma\in \Gamma$ and
  $m=M+d-k_0,\dots,M+d$ and $i=1,\dots,k_0$,
\end{outputlist}
such that the Ehrhart quasi-polynomial 
$$ E(\p,\ell,M; n) = \sum_ {x\in n \p\cap \lattice}
\frac{\langle\ell,x\rangle^M}{M!}\,
= \sum_{m=0}^{M+d} E_m(\p,\ell,M; \{n\}_q)\, n^m
$$
agrees in $n$-degree~$\geq M+d-k_0$ with the quasi-polynomial
\begin{displaymath}
  \sum_{\gamma\in \Gamma} \sum_{m=M+d-k_0}^{M+d} f^{\gamma,m}\Bigl(\{\zeta^{\gamma,m}_1 n\}_{q^{\gamma,m}_1},\dots, \{\zeta^{\gamma,m}_{k_0} n\}_{q^{\gamma,m}_{k_0}}\Bigr)\, n^m.
\end{displaymath}

\end{theorem}

\begin{remark}
  For $d\leq k_0$, the algorithm actually computes the complete Ehrhart
  quasi-polynomial, i.e., the coefficient functions $E_m(\p,\ell,M; \{n\}_q)$
  for $m = 0,\dots, M+d$.  The key point of our method, however, is to
  handle the case where $d > k_0$; then the non-trivial efficiently computable
  approximations come into play.
\end{remark}

\begin{remark}
  The specific form of the quasi-polynomial given by the theorem 
  gives a more precise period~$q_i$
  for the individual terms, rather than a period $q_s$ that is determined by
  the vertex.  The $q_i$ will always be divisors of~$q_s$.  Due to the
  projections into lattices in small dimension $\leq k_0$, 
  these periods can be much smaller than~$q_s$.
  In particular, the highest-degree coefficient~$E_{M+d}$ of course is a
  constant. 
  %% (in other words, all the periods $q_i^{\gamma,M+d}$ are~$1$).
  %%%%%% This parenthetical remark doesn't make sense for rational dilations
  %%%%%% any more??! --Matthias
\end{remark}
We will use the following lemma.
\begin{lemma}[Lemma 4 of \cite{baldoni-berline-deloera-koeppe-vergne:integration}] \label{truncated-product}
For every fixed number $D\in \N$, there exists a polynomial time
algorithm for the following problem.

\noindent Input: a number $M$  in unary encoding, a sequence of $k$
polynomials $P_j\in \Q[X_1,\dots, X_D]$ of total degree at most $M$,
in dense monomial representation.

\noindent Output: the product $P_1\cdots P_k$ truncated at degree
$M$.
\end{lemma}

We can now prove the theorem.

\begin{proof}[Proof of Theorem~\ref{powersoflinear-stepfun}]
  Because the polytope~$\p$ is simple, 
  we can use the primal--dual algorithm by Bremner, Fukuda, and Marzetta
  \cite[Corollary~1]{bremner-fukuda-marzetta-1998:primal-dual}, 
  to compute the inequality description
  (H-description) from the given V-description in polynomial time. 
  From the double description, 
  we can compute in polynomial
  time the description of the tangent cones $\smash{\c_{s_j}}$ for $j\in \CV$ by the primitive
  vectors $v_{s_j,1},\dots,v_{s_j,d}\in\Z^d$ such that $\c_{s_j} =
  \c(v_{s_j,1},\dots,v_{s_j,d})$.\smallbreak

  We now use formula~\eqref{eq:main-general-variant} of
  Theorem~\ref{main-general}, which gives (with $d_0=d-k_0$)
  \begin{multline}\label{overall-formula-for-coeff}
    E_m(p,\xi,M; \{n\}_q) \\= \sum_{s\in\CV(\p)} \frac{\langle\xi,s\rangle^{m}}{m!} \Bigl( \e^{-\langle\xi,\{n\}_{q_s}s\rangle} A_{\geq d_0}\bigl(\{n\}_{q_s} s +
      \c_s,\Lambda\bigr)(\xi)\Bigr)_{[-d+k]}
  \end{multline}
  for $m = M+d-k$, when $m\geq M+d-k_0$.  
  We compute this separately for each $k=0,\dots, k_0$, that is, $m =
  M+d-k_0,\dots, M+d$. 
  Let~$s+\c_s$ be one of these cones.
  By the algorithm of Theorem~\ref{th:problem1} and
  Remark~\ref{patchedgenfun-cone-polytime}, we compute the data describing the 
  parametric short formula~\eqref{problem1-shortformula-variant} for $A_{\geq d_0}\bigl(\{n\}_{q_s} s +
  \c_s,\Lambda\bigr)(\xi)$.  
  We then consider one of the summands of  
  $$
  \frac{\langle\xi,s\rangle^m}{m!}
  \Bigl(
  \e^{-\langle\xi,\{n\}_{q_s}s\rangle} A_{\geq d_0}\bigl(\{n\}_{q_s} s +
  \c_s,\Lambda\bigr)(\xi)\Bigr)_{[-d+k]}$$ 
  at a time.  
  Here $\e^{-\langle\xi,\{n\}_{q_s}s\rangle}$ and
  the term $\e^{\langle\xi,\{n\}_{q_s}s\rangle}$
  from~\eqref{problem1-shortformula-variant} cancel, and thus each summand 
  takes the form
 \begin{equation}\label{one-summand}
    \Bigl(\frac{\langle\xi,s\rangle^{M+d-k}}{(M+d-k)!}\Bigr) \Bigl(
    \prod_{i\in I^c} T\bigl(\tau_i(n) , \bigl\langle \xi,w_i\bigr\rangle \bigr)\Bigr)_{[k]}
    \Bigl(
    \frac{1} {\prod_{i=1}^d\bigl\langle  \xi,w_i\bigr\rangle} \Bigr)
  \end{equation}
  where
  \begin{equation}
    \tau_i(n) := \bigl\{ -\bigl\langle \eta _i,
    s\bigr\rangle\{ n \}_{q_s} \bigr\} \quad 
    \text{for $i\in I^c$}.
  \end{equation}
  Let $q_i\in\N$ be the smallest positive integer such
  that $q_i\langle-\eta _i,s\rangle \in\Z$.  
  Then $q_i$ is a divisor of the number~$q_s$ associated with the vertex~$s$, because $\eta _i\in\Z^d$.
  Then
  \begin{displaymath}
    \tau_i(n) = \tfrac1{q_i} \bigl\{ \zeta_i \{ n \}_{q_s} \bigr\}_{q_i}
    \quad\text{with}\quad \zeta_i = q_i \langle-\eta _i,s\rangle \in \Z.
  \end{displaymath}
  Since $q_i$ is a divisor of $q_s$, this simplifies to 
  \begin{equation}
    \tau_i(n) = \tfrac1{q_i} \bigl\{ \zeta_i n \bigr\}_{q_i},
  \end{equation}
  where of course $\zeta_i$ can be reduced modulo~$q_i$ as well because $n$ is
  assumed to be an integer.
  We now treat $r_i := \{ \zeta_i n \}_{q_i} \in \N$ as symbolic
  variables.  

  In order to %% pick the terms of $\xi$-degree $-d+k$ for $k=0,\dots,k_0$
  %% and
  evaluate~\eqref{one-summand} at
  $\xi=\ell$, we note the following. The first factor is holomorphic in~$\xi$
  and homogeneous of $\xi$-degree~$m=M+d-k$, the second factor is holomorphic in~$\xi$ and homogeneous of degree $k$, and the
  third factor is homogeneous of $\xi$-degree~$-d$.
            If $\la \ell,w_i\ra = 0$ for some $i$, we cannot
  just substitute~$\xi=\ell$ in the formula.  
  Instead we use a perturbation.
  In polynomial time, we can compute a rational vector~$\ellpert\in\Q^d$ such
  that  $\la \ellpert,w_i\ra \neq 0$ for all
  vectors~$w_i$ with $\la \ell,w_i\ra = 0$.  It is important that we choose
  the same vector once and for all computations with all cones and summands.  
  
  We then set $\xi=t (\ell+\epsilon\ellpert)$, where $t $ and
  $\epsilon$ are treated as symbolic variables.    Here the exponent of the variable~$t $ keeps
  track of the $\xi$-grading.
  We then do computations with
  truncated series in $\Q[r_i : i\in I^c][t ^{\pm1},\epsilon^{\pm1}]$.
  We note that this is a polynomial ring in a constant number of variables
  only, because $|I^c| $ is bounded above by the constant~$k_0$.  Thus Lemma~\ref{truncated-product}
  gives us a polynomial-time algorithm for multiplying the series.  Then
  \eqref{one-summand} can be written as:
  \begin{equation}\label{one-summand-with-t-factored}
    \frac{\langle\ell+\epsilon\ellpert,s\rangle^{m}}{m!} \cdot \Bigl(
    \prod_{i\in I^c} T\bigl(\tau_i(n) , \bigl\langle t(\ell+\epsilon\ellpert),w_i\bigr\rangle \bigr)
    \Bigr)_{[k]}\cdot\frac{1} {\prod_{i=1}^d\bigl\langle
      \ell+\epsilon\ellpert,w_i\bigr\rangle} \cdot t^{M-k} ,
  \end{equation}
  where the subscript~$[k]$ now means to take the term of $t$-degree~$k$.
  In the end we are 
  interested in the coefficient of the term~$t ^{M} \epsilon^0$.

  Expanding the factors of~\eqref{one-summand-with-t-factored} gives the following contributions, all of which can be
  written down in polynomial time.    First of all, the rational terms $\la
  \ell+\epsilon\ellpert, w_i \ra^{-1}$ give the
  following contribution. If $\la \ell, w_i \ra=0$, we simply get~
  \begin{equation}\label{ratl-term-with-negative-epsilon-degree}
    \frac1{\la \ell+\epsilon\ellpert, w_i \ra}
    = \frac1{\la \ellpert, w_i \ra} \epsilon^{-1}.
  \end{equation}
  If $\la \ell, w_i \ra\neq 0$, we get the geometric series in~$\epsilon$
  \begin{displaymath}
    \frac1{\la \ell+\epsilon\ellpert, w_i \ra}
    = \frac{1}{\la \ell, w_i\ra}\sum_{u=0}^{\infty} \left(-\frac{\la \ellpert,
        w_i\ra}{\la \ell, w_i\ra}\right)^u \epsilon^u.
  \end{displaymath}
  The
  first and second terms in~\eqref{one-summand-with-t-factored} are holomorphic, thus the only negative degrees
  in~$\epsilon$ come from the rational
  terms~\eqref{ratl-term-with-negative-epsilon-degree}.  Let~$U$ be the number
  of vectors~$w_i$ that are orthogonal to $\ell$;
  then~$\epsilon^{-U}$ is the lowest negative degree.  Note that $U\leq d$.
  Since we wish to find the term of $\epsilon$-degree~0, we can truncate all series
  after~$\epsilon$-degree~$U$:
  \begin{equation}\label{ratl-term-without-negative-epsilon-degree-truncated}
    \frac1{\la \ell+\epsilon\ellpert, w_i \ra}
    = \frac{1}{\la \ell, w_i\ra}\sum_{u=0}^{U} \left(-\frac{\la \ellpert,
        w_i\ra}{\la \ell, w_i\ra}\right)^u  \epsilon^u
    +  o_\epsilon(\epsilon^{U}).
  \end{equation}
  We expand the first factor of~\eqref{one-summand-with-t-factored} as follows.
  \begin{equation}\label{powerlinform-expanded}
    \frac{\langle \ell+\epsilon\ellpert,s\rangle^m}{m!} = 
    \sum_{u=0}^{\min\{m,U\}} \binomial(m,u) \langle\ell,s\rangle^{m-u}
    \langle\ellpert,s\rangle^u \epsilon^u
    + o_\epsilon(\epsilon^U).
  \end{equation}
  Now we consider the holomorphic terms
  \begin{align}
    &T(\tau_i(n), \langle t (\ell+\epsilon\ellpert),w_i\rangle)
%%     = e^{\tau_i (\langle
%%       t (\ell+\epsilon\ellpert),w_i\rangle)}
%%     \frac {\langle t (\ell+\epsilon\ellpert),w_i\rangle}{1-e^{\langle
%%         t (\ell+\epsilon\ellpert),w_i\rangle}}
    \notag\\
    &\quad= -\sum_{j=0}^\infty \frac1{j!} B_j(\tau_i)\, \bigl\langle
    t (\ell+\epsilon\ellpert),w_i\bigr\rangle^j \notag\\
    &\quad= -\sum_{j=0}^{k_0} \frac1{j!} B_j\bigl(\tfrac1{q_i} r_i\bigr)
    \left(\sum_{u=0}^{\min\{j,U\}} \binomial(j,u) \langle\ell, w_i\rangle^{j-u} \langle \ellpert,w_i\rangle^{u}
    \epsilon^{u}\right) t ^j\notag\\
  & \qquad\qquad\qquad
  + o_t (t ^{k_0}) + o_\epsilon(\epsilon^U).\label{todd-expanded}
  \end{align}
  The Bernoulli polynomials $B_j(\tau_i)$ of degree $j\leq k_0$ that appear in this formula can be efficiently expanded in
  polynomial time using recursion formulas.  We remark that the variables~$r_i$ appear
  with a degree that is at most that of~$t$. 
  Using Lemma~\ref{truncated-product}, we multiply the truncated
  series~\eqref{todd-expanded} for $i\in I^c$ in
  $\Q[r_i : i\in I^c][t][\epsilon]$, truncating 
  in each step after $t ^{k_0}$ and $\epsilon^{U}$.  We thus obtain the second
  factor of~\eqref{one-summand-with-t-factored}, 
  \begin{equation}\label{todd-product-one-coeff}
    \Bigl(
    \prod_{i\in I^c} T\bigl(\tau_i(n) , \bigl\langle t(\ell+\epsilon\ellpert),w_i\bigr\rangle \bigr)
    \Bigr)_{[k]}    \quad\text{for all $k=0,\dots,k_0$,}
  \end{equation}
  as a truncated series in $\Q[r_i : i\in I^c][\epsilon]$. 

  Then we multiply the truncated series
  \eqref{ratl-term-with-negative-epsilon-degree},
  \eqref{ratl-term-without-negative-epsilon-degree-truncated},
  \eqref{todd-product-one-coeff}, and~\eqref{powerlinform-expanded}
  in polynomial time, truncating
  in each step after $\epsilon^{U}$, using
  Lemma~\ref{truncated-product}.  
  In the end, we read out the coefficient of~$\epsilon^0$ as a polynomial in~$\Q[r_i : i\in I^c]$.
  Then we substitute for~$r_i$.
  Collecting these terms gives the formula for the Ehrhart coefficient~$E_m(\p,\xi,M; \{n\}_q)$.
\end{proof}

{\small\begin{example}
  Let us give a small example of the output of our algorithm for
  $E_m(\p,\ell,M,\{n\}_q)$, when $\p$ is the simplex in $\R^5$ with vertices:
  $$(0, 0, 0, 0,0),\; (\tfrac{1}{2}, 0, 0, 0,0),\; (0,\tfrac{1}{2} , 0,
  0,0),\; (0, 0, \tfrac{1}{2}, 0,0),\; (0, 0, 0, \tfrac{1}{6},0),\; (0, 0, 0,0, \tfrac{1}{6}).$$
We consider the linear form $\ell$ on $\R^5$ given by the scalar product with
$(1,1,1,1,1)$.

If $M=0$, the coefficients of $E_m(\p,\ell,M=0;\{n\}_q)$ are just the
coefficients of the unweighted Ehrhart quasi-polynomial $S(n\p,1)$. We obtain
\begin{multline*}
  S(n\p,1)=\frac{1}{34560}n^5+\Bigl(\frac{5}{3456}-\frac{1}{6912}\{n\}_2\Bigr)
  n^4\\
  +\Bigl(\frac{139}{5184}-\frac{5}{864}\{n\}_2+\frac{1}{3456}(\{n\}_2)^2\Bigr)n^3+\cdots.
\end{multline*}
Now if $M=1$, all integral points $(x_1,x_2,x_3,x_4,x_5)$ are weighted with
the function $h(x)=x_1+x_2+x_3+x_4+x_5$, and we obtain
\begin{multline*}
  S(n\p,h)=\frac{11}{1244160}n^6+\Bigl(\frac{19}{41472}-\frac{11}{207360}\{n\}_2\Bigr)n^5\\
  + \Bigl(\frac{553}{62208}-\frac{95}{41472}\{n\}_2+\frac{11}{82944}(\{n\}_2)^2\Bigr)n^4+\cdots.
\end{multline*}

We can remark that although $q=6$ is the smallest integer such that $q\p$ is a
lattice polytope, only periodic functions of $n$ mod $2$ enter in the top
three Ehrhart coefficients. This is indeed conform to the known periodicity
properties of the Ehrhart coefficients.
\end{example}}

As a corollary, simply by evaluating the step polynomials, we obtain the
following result, which directly extends the complexity result from Barvinok's
paper to the weighted case.

\begin{theorem}[Evaluation of the Ehrhart coefficients for a given dilation class
  $\{n\}_q$] 
  \label{ehrhart-barvinok-style}
  For every fixed number ~$k_0 \in\N$, there exists a polynomial-time algorithm
  for the following problem.

\noindent  Input:
\begin{inputlist}
\item a number~$ d\in\N$ in unary encoding, with $d\geq k_0$,

\item a finite index set $\CV$,

\item a simple polytope~$\p$, given by its vertices, 
  rational vectors $\ve s_j \in \Q^d$ for $j\in\CV$ %% , and their tangent
%%   cones $\c_{s_j} = \c(v_{s_j,1},\dots,v_{s_j,d}) \subset \R^d$, represented
%%   by the primitive vectors $v_{s_j,1},\dots,v_{s_j,d}\in\Z^d$
  in binary
  encoding,
\item a rational vector $\ell \in \Q^d$ in binary encoding,
\item a number~$M \in \N$ in unary encoding,
\item a number~$n$ in binary encoding,
\end{inputlist}
\noindent Output, in binary encoding,
\begin{outputlist}
\item a positive integer $q\in\N$ such that $q\p$ is a lattice polytope and
\item the numbers $E_m(\p,\ell,M;\{ n \}_q)$ for $m=M+d-k_0,\dots,M+d$.
\end{outputlist}
\end{theorem}
\begin{remark}
  A direct algorithm for computing $E_m(\p,\ell,M;\{ n \}_q)$ for just one dilation
  class~$\{ n \}_q$ could of course use the values $r_i = \{ \zeta_i n
  \}_{q_i} \in\Z$ rather than symbolic variables~$r_i$ and would therefore
  only need to do calculations with truncated series in the two-variable ring
  $\Q[t ^{\pm1},\epsilon^{\pm1}]$.
\end{remark}

Via the decomposition of polynomials into powers of linear
forms, which is, as
discussed in \cite{baldoni-berline-deloera-koeppe-vergne:integration},
polynomial-time  under suitable hypotheses,
%% and the
%% discussion of the computational complexity of simple polytopes at the
%% beginning of this section, 
we obtain the following corollary.

\begin{corollary}
  \label{ehrhart-barvinok-style-general}
  For every fixed number ~$k_0 \in\N$, there exist polynomial-time algorithms
  for the following problems.

\noindent  Input:
\begin{inputlist}
\item a number~$ d\in\N$ in unary encoding, with $d\geq k_0$,
\item a simple rational polytope~$\p\subset\R^d$, given by its vertices in binary
  encoding,
  %%% NOTE: We cannot accept it to be given by linear inequalities!!
\item a number~$M$ in unary encoding,
\item a polynomial~$h$ of degree~$\leq M$ which is given either as
  \begin{enumerate}[\rm(a)]
  \item a power of a linear form, or
  \item a sparse polynomial where each monomial only depends on a fixed number of variables, or
  \item a sparse polynomial of fixed total degree, 
  \end{enumerate}
\item a number~$n$ in binary encoding,
\end{inputlist}
\noindent Output, in binary encoding,
\begin{outputlist}
\item a positive integer $q\in\N$ such that $q\p$ is a lattice polytope and
\item the numbers $E_m(\p, h; \{ n \}_q)$ for $m=M+d-k_0,\dots,M+d$.
\end{outputlist}
\end{corollary}

%\clearpage
\section{Experiments} \label{experiments}

We implemented the algorithms in \emph{Maple}, for the unweighted case and assuming that
the input were lattice simplices of full dimension (in this case the quasi-polynomial becomes a polynomial). This assumption was made for simplicity
of output in the calculation and because available software to verify the
results (e.g., \emph{LattE macchiato}~\cite{latte-macchiato}) cannot
compute with weights. In addition, already the problem of computing Ehrhart polynomials for lattice simplices has received attention by
many researchers and it is non-trivial (see e.g., the references in \cite{beifangchen:2002}).
After checking simple low-dimensional examples by hand, we set up automatic scripts for generating random tests.
The simplices generated had vertex coordinates drawn uniformly at random from $\{-99,\dots,99\}$. We  timed the speed of the procedure to compute the top three Ehrhart coefficients in 50 random simplices per dimension and recorded the average time of computation.
We compared with the computation of the \emph{full} Ehrhart polynomials
using the state-of-the-art algorithms implemented in \emph{LattE macchiato}
\cite{latte-macchiato}; see Table~\ref{tab:lattice-simplices}.

\begin{table}[t]
\caption{Computation times for Ehrhart polynomials of random lattice simplices}
\label{tab:lattice-simplices}
\begin{center}
\def~{\hphantom0}
\begin{tabular}{ c  c  c  c  c }
\toprule
& \multicolumn{4}{c}{Average runtime (CPU seconds)} \\
\cmidrule{2-5}
& \multicolumn{3}{c}{Full (\emph{LattE macchiato})} &  \\
\cmidrule{2-4} 
Dimension &  \quad Dual\quad & \quad Primal\quad & Primal${}_{1000}$ & \smash[t]{\begin{tabular}[b]{c}Top 3\\ (new code) \end{tabular}} \\ 
\midrule
        3       & ~0.16     & ~~0.10  & ~~0.04 & ~~~1.12    \\[1ex]
        4       & 28.00       & ~~4.68 & ~~0.28 & ~~~4.31    \\[1ex]
        5       & & 317.5~ & ~~5.8~         & ~~13.4~   \\[1ex]
        6       & &        & 198.0~      & ~~37.4~   \\[1ex]
        7       & &        &             & ~103\hphantom{.00}  \\[1ex]
        8       & &        &             & ~294\hphantom{.00}  \\[1ex]
        9       & &        &             & ~393\hphantom{.00}  \\[1ex]
        10      & &        &             & 1179\hphantom{.00} \\[1ex]
        11      & &        &             & 1681\hphantom{.00} \\
\bottomrule
\end{tabular}
\end{center}
\end{table}

In the table, \emph{Dual} refers to an implementation of Barvinok's
decomposition of the duals of the tangent cones into unimodular
cones, as implemented first in \emph{LattE}~\cite{latte1}, and which is still
the default method in \emph{LattE macchiato}.\footnote{The LattE macchiato command is \texttt{count
    --ehrhart-polynomial}.}
\emph{Primal} refers to a primal variant of Barvinok's
decomposition described in~\cite{koeppe:irrational-barvinok}; it is more
efficient for these examples because the determinants of the dual cones are
much larger.\footnote{The  command is \texttt{count
    --ehrhart-polynomial --irrational-primal}.}  We remark that our implementation of the new
algorithm in Maple also uses a primal variant of Barvinok's decomposition to
unimodular cones, which was introduced in \cite{Brion1997residue}.  
Thus the new code should be compared to the runtimes listed in column
\emph{Primal}.  Finally, \emph{Primal}${}_{1000}$ refers to a variant in which
Barvinok's decomposition is stopped when a cone has a determinant at
most~$1000$; then the points in the fundamental parallelepipeds are
enumerated.\footnote{The command used is \texttt{count
    --ehrhart-polynomial --irrational-primal --maxdet=1000}.}  
%%% Note that "--ehrhart-polynomial" implies "--exponential", so I omitted
%%% that option here. --Matthias

All computations were stopped if unfinished after 30 minutes, thus the table ends at dimension 11 because all
randomly generated examples we tried in dimension 12 took more than 30 minutes
of calculation. 
The computation times are given in CPU seconds on a computer with AMD Opteron 880 processors
running at 2.4\,GHz. 
\smallbreak

In conclusion, the experiments indicate that the algorithms presented here can
lead to dramatic improvements upon the computation of full Ehrhart polynomials.
The fact that, for very low dimensions, the implementation is slower than
\emph{LattE macchiato}, is explained by the choice of \emph{Maple} as an
implementation language.  \emph{Maple} is an interpreted system, which is much slower
than C++, the implementation language of \emph{LattE macchiato}. 
We expect that the speedups of~\emph{Primal}${}_{1000}$ compared to \emph{Primal}, which were
first documented in~\cite{koeppe:irrational-barvinok}, will also be obtained in
a refined implementation of our new algorithms.  

The implementation is available at~\cite{topehrhart-accompanying-programs}.

%%\clearpage

\section*{Acknowledgments}

 This article is part of a research which was made possible by
 several meetings of the authors,
at the Centro di Ricerca Matematica Ennio De Giorgi of the Scuola
Normale Superiore, Pisa in 2009,  in a SQuaRE program at the
American Institute of Mathematics, Palo Alto, in July 2009 and
September 2010, and in the Research in Pairs program at
Mathematisches Forschungsinstitut Oberwolfach in March/April 2010.
The support of all three institutions is gratefully acknowledged.

V.~Baldoni was partially supported by the Cofin 40\%, MIUR.
J. De Loera was partially supported by grant DMS-0914107 of the National
Science Foundation.
M.~K\"oppe was partially supported by grant DMS-0914873 of the National
Science Foundation.

The authors wish to thank two undergraduate students at UC Davis,
Brandon Dutra and Gregory Pinto, for their diligent help with the
computational experiments and testing.

\bibliography{../biblio}
\bibliographystyle{../amsabbrv}

\end{document}

%% file: triangle528.pdf_t
\begin{picture}(0,0)%
\includegraphics{triangle528.pdf}%
\end{picture}%
\setlength{\unitlength}{4144sp}%
\begingroup\makeatletter\ifx\SetFigFontNFSS\undefined%
\gdef\SetFigFontNFSS#1#2#3#4#5{%
  \reset@font\fontsize{#1}{#2pt}%
  \fontfamily{#3}\fontseries{#4}\fontshape{#5}%
  \selectfont}%
\fi\endgroup%
\begin{picture}(2788,1983)(-943,796)
\put(451,1289){\makebox(0,0)[lb]{\smash{{\SetFigFontNFSS{12}{14.4}{\familydefault}{\mddefault}{\updefault}{\color[rgb]{0,0,0}$n=6$}%
}}}}
\end{picture}%

%% file: triangle528.pstex_t
\begin{picture}(0,0)%
\includegraphics{triangle528.pstex}%
\end{picture}%
\setlength{\unitlength}{4144sp}%
\begingroup\makeatletter\ifx\SetFigFontNFSS\undefined%
\gdef\SetFigFontNFSS#1#2#3#4#5{%
  \reset@font\fontsize{#1}{#2pt}%
  \fontfamily{#3}\fontseries{#4}\fontshape{#5}%
  \selectfont}%
\fi\endgroup%
\begin{picture}(2788,1983)(-943,796)
\put(451,1289){\makebox(0,0)[lb]{\smash{{\SetFigFontNFSS{12}{14.4}{\familydefault}{\mddefault}{\updefault}{\color[rgb]{0,0,0}$n=6$}%
}}}}
\end{picture}%

%% file: projected-lattice.pdf_t
\begin{picture}(0,0)%
\includegraphics{projected-lattice.pdf}%
\end{picture}%
\setlength{\unitlength}{4144sp}%
\begingroup\makeatletter\ifx\SetFigFontNFSS\undefined%
\gdef\SetFigFontNFSS#1#2#3#4#5{%
  \reset@font\fontsize{#1}{#2pt}%
  \fontfamily{#3}\fontseries{#4}\fontshape{#5}%
  \selectfont}%
\fi\endgroup%
\begin{picture}(2904,2184)(-326,197)
\put(676,1649){\makebox(0,0)[b]{\smash{{\SetFigFontNFSS{12}{14.4}{\familydefault}{\mddefault}{\updefault}{\color[rgb]{0,0,.56}$v_2$}%
}}}}
\put(856,601){\makebox(0,0)[b]{\smash{{\SetFigFontNFSS{12}{14.4}{\familydefault}{\mddefault}{\updefault}{\color[rgb]{0,0,.56}$v_1$}%
}}}}
\end{picture}%

%% file: projected-lattice.pstex_t
\begin{picture}(0,0)%
\includegraphics{projected-lattice.pstex}%
\end{picture}%
\setlength{\unitlength}{4144sp}%
\begingroup\makeatletter\ifx\SetFigFontNFSS\undefined%
\gdef\SetFigFontNFSS#1#2#3#4#5{%
  \reset@font\fontsize{#1}{#2pt}%
  \fontfamily{#3}\fontseries{#4}\fontshape{#5}%
  \selectfont}%
\fi\endgroup%
\begin{picture}(2904,2184)(-326,197)
\put(676,1649){\makebox(0,0)[b]{\smash{{\SetFigFontNFSS{12}{14.4}{\familydefault}{\mddefault}{\updefault}{\color[rgb]{0,0,.56}$v_2$}%
}}}}
\put(856,601){\makebox(0,0)[b]{\smash{{\SetFigFontNFSS{12}{14.4}{\familydefault}{\mddefault}{\updefault}{\color[rgb]{0,0,.56}$v_1$}%
}}}}
\end{picture}%